\documentclass[a4paper,11pt,reqno]{amsart}
\usepackage[alphabetic,msc-links]{amsrefs} 

\usepackage[utf8]{inputenc} 
\usepackage[T1]{fontenc} 

\usepackage{amsmath} 
\usepackage{amsthm} 
\usepackage{amssymb} 
\usepackage{amscd} 

\usepackage{emptypage}
\usepackage{enumitem} 
\usepackage{multicol}
\usepackage{mathtools} 
\usepackage{mathrsfs} 
\usepackage{hyperref} 
\hypersetup{
	linktocpage = true,
	colorlinks=true,
	linkcolor=red,
}
\usepackage{esint} 

\usepackage{graphicx} 
\usepackage{subcaption}
\usepackage[margin=2cm]{geometry}

\usepackage{pgf,tikz,pgfplots}
\usetikzlibrary{arrows}
\usetikzlibrary{intersections}
\usetikzlibrary{fadings}
\usetikzlibrary{shapes}

\usepackage{cleveref,autonum}

\usepackage[colorinlistoftodos]{todonotes} 
\newtheorem{theorem}{Theorem}[section]
\newtheorem*{theorem*}{Theorem}
\newtheorem{proposition}[theorem]{Proposition}
\newtheorem{lemma}[theorem]{Lemma}

\theoremstyle{definition}

\theoremstyle{remark}
\newtheorem{remark}[theorem]{Remark}
\newtheorem*{remark*}{Remark}

\numberwithin{equation}{section}

\newcommand{\con}[1]{\mathbb{#1}}
\newcommand{\R}{\con{R}} 

\usepackage{euscript}
\usepackage{relsize}
\DeclareMathAlphabet{\mathpzc}{OT1}{pzc}{m}{it}
\DeclareMathAlphabet\euscr{T1}{qzc}{m}{n}

\makeatletter
\newcommand{\leqnomode}{\tagsleft@true\let\veqno\@@leqno}
\newcommand{\reqnomode}{\tagsleft@false\let\veqno\@@eqno}
\makeatother

\renewcommand{\d}{\,\mathrm{d}} 

\newcommand\beqc[1]{\left\{\begin{array}{#1}}
	\newcommand\eeqc{\end{array} \right.}

\def\bmatrix{\begin{pmatrix}}
	\def\ematrix{\end{pmatrix}}
\DeclareMathOperator{\Tr}{Tr}

\DeclareMathOperator{\dist}{dist}

\let\div\relax
\DeclareMathOperator{\div}{div}

\renewcommand{\le}{\leqslant}
\renewcommand{\leq}{\leqslant}

\renewcommand{\geq}{\geqslant}

\numberwithin{equation}{section}

\def\dfb{\mathfrak{d}}

\title[Stable cones in the Alt-Phillips free boundary problem]{Stable cones in the Alt-Phillips free boundary problem}

\author{Aram Karakhanyan}
\address{Aram Karakhanyan:
	School of Mathematics, The University of Edinburgh,
	Peter Tait Guthrie Road, EH9 3FD Edinburgh, UK}
\email{aram6k@gmail.com}

\author{Tomás Sanz-Perela}
\address{T. Sanz-Perela:
	Departament de Matemàtiques i Informàtica, Universitat  de Barcelona, Gran Via de les Corts Catalanes 585, 
	08007 Barcelona, Spain}
\email{tomas.sanz.perela@ub.edu}

\begin{document}
	
\begin{abstract}
In this paper we prove a classification result for axially symmetric one phase minimizers of the Alt-Phillips free boundary problem in dimensions 3, 4, and 5.
To accomplish this, we establish a stability inequality that extends the one for the Alt-Caffarelli problem.
\end{abstract}
	
\maketitle

\tableofcontents

\section{Introduction}
For a domain $\Omega\subset\R^n$ with locally Lipschitz boundary, the Alt-Phillips functional is defined as 
\begin{equation}
	\label{Eq:AltPhillipsFunctional}
	E_\gamma^\mathrm{AP} [u] := \int_{\Omega} \big( |\nabla u|^2 + u^\gamma \chi_{\{u>0\}} \big) \d x,
\end{equation}
where $\gamma\in(0, 2)$ and $\chi_E$ denotes the characteristic function of a set $E\subset \R^n$. 
We consider the minimizers of $E_\gamma^\mathrm{AP} [\cdot]$ over the linear space 
$\mathbb K_{u_0}:=\{v\in H^1(\Omega): v-u_0\in H^1_0(\Omega)\}$, where $0\le u_0\in H^1(\Omega)$ ---which forces any minimizer to be nonnegative. 
Our objective in this paper is to classify global minimizers (or, more generally, stable critical points) that are axially symmetric.
Through the paper, we will say that $u$ is a \textit{global minimizer} if it is a minimizer over $\mathbb K_{u}$ for any bounded domain $\Omega\subset \R^n$, and we will say that it is \textit{stable} if the second variation of $E_\gamma^\mathrm{AP}[\cdot]$ at $u$ is nonnegative (see~\Cref{Prop:ExpansionAndStabilityGeneral} below).

The problem above appears in population dynamics \cite{GurtinMacCamy}, where the density  $\rho$
of the population is governed by the porous medium equation
\begin{equation}
	\rho_t=\Delta \phi(\rho)+\sigma(\rho),
\end{equation}
where $\sigma(\rho)$ represents the population supply due to births and deaths, and $\phi$ is a nonlinear function of $\rho$. 
For the steady-case problem with linear supply function modeling a death dominant rate (otherwise the whole domain becomes populated, i.e., $\rho>0$ everywhere), when $\phi$ is a power function  solutions can be obtained (after a change of variable) as minimizers of $E_\gamma^\mathrm{AP} [\cdot]$.

As usual in this type of free boundary problems, in order to understand the regularity of the free boundary $\partial\{u>0\}\cap \Omega$, after a blow-up one is led to study global homogeneous minimizers of $E_\gamma^\mathrm{AP} [\cdot]$. 
In our case, there are two regimes: for $\gamma\in [1, 2)$ the problem is akin to the obstacle problem ($\gamma = 1$), namely the blow-up limits are convex functions (see \cite{AltPhillips,Bonorino}), whereas for $\gamma\in(0, 1)$ it shares the features of the classical Alt-Caffarelli (or Bernoulli) problem, corresponding to $\gamma=0$ (see \cite[Lemma~4.3]{AltPhillips} and \cite[Remark~ 3.4]{PhillipsCPDE}). 
Even for the case $\gamma=0$ the classification of global minimizers in dimensions $n=5, 6$ is still open (see \cite{CaffarelliJerisonKenig, JerisonSavin, DeSilvaJerison}). 
For the general functional $E_\gamma^\mathrm{AP} [\cdot]$ the classification of global solutions is only known in dimension $n=2$ (see \cite{AltPhillips}) and thus for $\gamma\in(0, 1)$ and $n\geq 3$ it is an outstanding open problem.
In this paper, we settle it in dimensions $n=3,4,5$ in some range of $\gamma$ for stable axially symmetric solutions, which are the natural counterexamples in the context of this problem (see \cite{DeSilvaJerison}).

The key tool to classify global minimizers for $\gamma=0$ is the stability inequality 
\begin{equation}
	\label{Eq:StabilityOnePhase-Intro}
\int_{\partial \{u>0\}} H\varphi^2 \d \sigma \leq \int_{\R^n}|\nabla\varphi|^2 \d x \quad \text{ for all } \varphi \in C^1_c(\R^n),
\end{equation}
where $H$ is the mean curvature\footnote{\label{Footnote:MeanCurvature}Along this paper, the mean curvature of the free boundary is the mean curvature of the boundary of $\{u=0\}$  (assuming $\{u=0\}$ with nonempty interior), oriented in the direction of $\nabla u/|\nabla u|$.
} 
of the free boundary $\partial \{u>0\}$ (assumed to be smooth in the support of the test function $\varphi$); this inequality is a characterization of the nonnegativeness of the second variation of $E_0^\mathrm{AP} [\cdot]$ at $u$.
For $n=3$, \eqref{Eq:StabilityOnePhase-Intro} implies that the free boundary of any global homogeneous minimimizer consists of a single convex cone, and thus since homogeneous harmonic functions in cones have homogeneity which decreases as the domain increases, this yields that the free boundary is a hyperplane; see~\cite{CaffarelliJerisonKenig}.
The stability inequality \eqref{Eq:StabilityOnePhase-Intro} has also been used in \cite{FernandezRealRosOton2019global} to classify axially symmetric solutions (up to dimension $n=5$), and in~\cite{DeSilvaJerison} (indirectly, with a non-variational approach) to classify global homogeneous minimizers in dimensions $n=3,4$.

Our main technical result is contained in  \Cref{Th:StabilityCondition}, 
where 
we derive a stability inequality for the Alt-Phillips functional \eqref{Eq:AltPhillipsFunctional}. 
More precisely, we show that if $u$ is a nonnegative stable critical point of $E_\gamma^\mathrm{AP} [\cdot]$ in $\Omega\subset \R^n$, then
\begin{equation}
	\label{Eq:StabilityAP1-intro}
	\dfrac{2-\gamma}{2} \dfrac{\gamma}{2}  
	\int_{\{u>0\} \cap \Omega} u^{\gamma} \dfrac{u^\gamma - |\nabla u|^2}{u^2}\varphi^2  \d x 
	\leq \int_{\{u>0\} \cap \Omega} u^{\gamma} |\nabla \varphi |^2  \d x  \quad \text{ for all } \varphi \in C^1_c(\Omega).
\end{equation}	
Let us mention that from here one recovers the stability inequality for $\gamma=0$ in the limit (see \Cref{Sec:AlphaToZero}), and thus 
\eqref{Eq:StabilityAP1-intro} is more general than
\eqref{Eq:StabilityOnePhase-Intro}. 
Moreover, it cannot be derived, even formally, from the stability condition for semilinear equations of the form $-\Delta u=f(u)$; see \Cref{Remark:ExteriorVariation}.

Our main classification result hinges on the stability 
inequality \eqref{Eq:StabilityAP1-intro} derived in \Cref{Th:StabilityCondition} and it is about  
classification of global stable solutions that are axially symmetric.
These are solutions with rotational symmetry with respect to some axis (see the precise definition in \Cref{Sec:AxSymSol}).

\begin{theorem}
	\label{Th:AxiallySymm}
	Let $n \geq 3$ and $\gamma \in (0,2/3)$, and let $u\geq 0$ be a global stable critical point of $E_\gamma^\mathrm{AP} [\cdot]$ that is axially symmetric.
	Assume that the dimension satisfies
	\begin{equation}
		\label{Eq:DimensionConstrain}
		2 + (1 - \sqrt{1 - \alpha})^2 < n < 2 + (1 + \sqrt{1 - \alpha})^2, \quad \text{where } \alpha=\frac{2\gamma}{2-\gamma}.
	\end{equation} 
	Then, $u$ is one-dimensional.
\end{theorem}
\begin{center}
	\begin{figure}[h]
		\begin{tikzpicture}[xscale=4,yscale=1, domain=0:1, samples=200]
			
			\draw[->] (-0.2,0) -- (1.1,0) node[right] {$\alpha=\frac{2\gamma}{2-\gamma}$};
			\draw[->] (0,0) -- (0,6.5) node[above] {$n$};
			\foreach \i in {0,0.25,...,1} {
				\draw (\i,0) -- (\i,-0.1) node[below] {$\i$};
			}
			
			\foreach \i in {1,2,...,6} {
				\draw (0,\i) -- (-0.05,\i) node[left] {$\i \,$ };
			}
			\draw[color=red, line width=1.5pt] 
			plot (\x,{2 + (1 - sqrt(1 - \x))^2}) ;
			\draw[color=red, line width=1.5pt] 
			plot (\x,{2 + (1 + sqrt(1 - \x))^2}) ;
			
			\draw[color=green, line width=1.8pt, fill=white, xscale=1/4, yscale=1] (0,2) circle[radius=3pt];
			
			\draw[color=green, line width=1.8pt] (0,3) -- (1,3);
			\draw[dashed] (1,3) -- (1,0);
			\draw[color=green, line width=1.8pt, fill=white, xscale=1/4, yscale=1] (4,3) circle[radius=3pt];

			\draw[color=green, line width=1.8pt] (0,4) -- (0.82842712474,4);
			\draw[dashed] (0.82842712474,4) -- (0.82842712474,0);
			\draw[color=green, line width=1.8pt, fill=white, xscale=1/4, yscale=1] (4*0.82842712474,4) circle[radius=3pt];
			
			\draw[color=green, line width=1.8pt] (0,5) -- (0.46410161513,5);
			\draw[dashed] (0.46410161513,5) -- (0.46410161513,0);
			\draw[color=green, line width=1.8pt, fill=white, xscale=1/4, yscale=1] (4*0.46410161513,5) circle[radius=3pt];
			
			\draw[color=green, line width=1.8pt, fill=white, xscale=1/4, yscale=1] (0,6) circle[radius=3pt];
			
		\end{tikzpicture}
		\caption{Visualization of the constraint \eqref{Eq:DimensionConstrain} on the dimension.}
		\label{Fig:Dimensions}
	\end{figure}
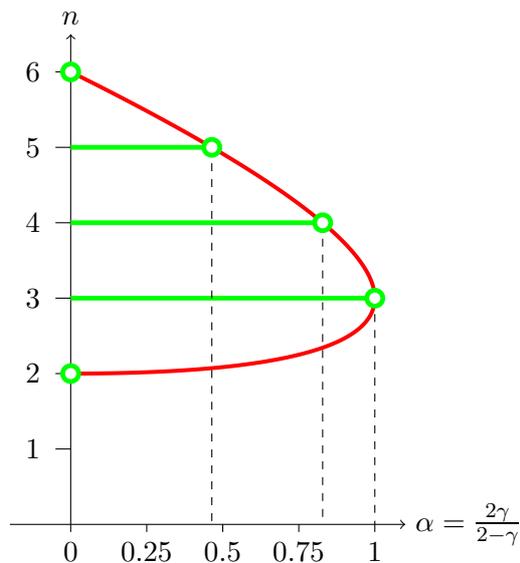
\end{center}

As a corollary, we obtain that, in the dimensions given by \eqref{Eq:DimensionConstrain}, if $u$ is an axisymmetric global stable solution whose zero phase $\{u=0\}$ has positive density, then $\{u=0\}$ must be a half-space.
We can visualize the dimension range \eqref{Eq:DimensionConstrain} in \Cref{Fig:Dimensions}.
Note that when $\gamma=0$, we get $2<n<6$, which are the dimensions obtained in \cite{FernandezRealRosOton2019global}, while for $\gamma=2/3$, i.e., $\alpha = 1$, the constraint \eqref{Eq:DimensionConstrain} becomes an empty condition.
The significance of the value $\gamma=2/3$ can be seen by considering $u=(v/\beta)^\beta$, with $\beta=\frac2{2-\gamma}$, in the Alt-Phillips functional.
We have
\begin{equation}
	E_\gamma^\mathrm{AP} [u] = \left(\frac{1}{\beta}\right)^{\gamma \beta} \int_{\{v>0\}\cap \Omega} v^{\gamma \beta} (|\nabla v|^2 + 1) \d x.
\end{equation}
and the weight $v^{\gamma \beta}$ is concave if 
$\gamma \beta<1$, which holds exactly for $\gamma\in(0, 2/3)$. 
Also in the proofs of regularity in \cite{RestrepoRosOton} the range $\gamma\in(0, 2/3)$ needs to be treated differently than $\gamma \geq 2/3$.
All this leads one to wonder whether there is a substantial change in the nature of the problem in the threshold value $\gamma = 2/3$.

The proof of \Cref{Th:AxiallySymm} follows a strategy which has a number of parallels with the regularity theory of minimal surfaces: in that case, a test function of the form $\varphi = c \eta$ is chosen in a stability inequality, where $c$ is the second fundamental form, which is a subsolution to some elliptic inequality, known as Simons' inequality; this key fact, combined with a cut-off argument choosing $\eta$ appropriately, leads to the classification of stable minimal cones up to dimension $n=7$ (see \cite{Simons}).
We follow a similar approach, choosing $\varphi = c \eta$ in \eqref{Eq:StabilityAP1-intro} where $c$ is a radial derivative of $u$ weighted with $u^{-\gamma/2}$ (see a more exhaustive discussion in \Cref{Sec:AxSymSol}).
Due to the singularity of $u^{-\gamma/2}$ at the free boundary and also the possible singularity of $\partial\{u>0\}$ at the origin, there are some difficulties that arise with this choice of $\varphi$.
To circumvent them, we use an approximation argument and a precise cancellation of some free boundary terms, which requires a fine control of $u$ and its derivatives near $\partial\{u>0\}$. 
This is the most delicate part of the proof, and it is contained in \Cref{Prop:Stabilityutau}.

\subsection*{Organization of the paper}

In \Cref{Sec:PreliminaryResults} we present some basic facts about the Alt-Phillips problem.
Then, we first obtain a general expansion of an energy functional in \Cref{Sec:GeneralExpansionAndStability}, and after that, we establish the stability condition for the Alt-Phillips problem \eqref{Eq:StabilityAP1-intro} in \Cref{Sec:StabilityInequalityn-2}. 
In \Cref{Sec:AlphaToZero} we briefly comment on the relation with the Alt-Caffarelli problem (i.e., the limit $\gamma \to 0$), and finally in  \Cref{Sec:AxSymSol} we prove our main result, \Cref{Th:AxiallySymm}.
At the end of the paper, there is a short appendix with two  elementary matrix results that we prove to make the article self-contained.

\subsection*{Acknowledgments}
Both authors are supported by the EPSRC grant EP/S03157X/1.
The second author is also supported by grants PID2020-113596GB-I00, PID2021-123903NB-I00, and RED2018-102650-T funded by MCIN/AEI/10.13039/501100011033 and by “ERDF A way of making Europe”, and by AGAUR Grant 2021 SGR 00087 (Catalunya).

\medskip

\section{Definitions and preliminary results}
\label{Sec:PreliminaryResults}

In this section we collect some results that will be used in the rest of the article.

\subsection{Regularity and behavior near the free boundary}

An important exponent in the context of the Alt-Phillips free boundary problem is
\begin{equation}
	\label{Eq:DefBeta}
	\beta := \dfrac{2}{2 - \gamma} \in (1, +\infty),
\end{equation}
which provides the precise behavior of solutions near the free boundary and gives the optimal regularity of solutions across $\partial \{u>0\}$.
Indeed, critical points of the Alt-Phillips functional \eqref{Eq:AltPhillipsFunctional} satisfy 
\begin{equation}
	\label{Eq:AltPhillipsEquation}
	\Delta u = \dfrac{\gamma}{2} u^{\gamma - 1} \quad \text{ in } \{u>0\}\cap \Omega,
\end{equation}
and also the free boundary condition
\begin{equation}
	\label{Eq:AltPhillipsFBC}
	|\nabla u| = 0 \quad \text{ on } \partial \{u>0\}\cap \Omega,
\end{equation}
and it was proved in \cite{AltPhillips} that any nonnegative minimizer of the previous free boundary problem is $C^{1, \beta-1}$ in $\Omega$.

More precisely, the exponent $\beta$ gives the exact power growth of solutions near the free boundary.
If we define the distance to the free boundary as
\begin{equation}
	\dfb := \dist(\cdot,  \{u=0\}\cap \Omega),
\end{equation}
then it follows (see \cite{AltPhillips}) that any nonnegative minimizer satisfies
\begin{equation}
	\label{Eq:BehaviorNearFbu}
	u \asymp \dfb^\beta = \dfb^\frac{2}{2-\gamma}.
\end{equation}
Here and through the paper we use the notation $w_1 \asymp w_2$ to mean that, for two functions $w_1$ and $w_2$, there exist two positive constants $ c \leq C$ such that $c w_2(x) \leq w_1(x) \leq C w_2(x)$ for all $x\in \{u>0\}$ sufficiently close to the free boundary $\partial \{u>0\}$.
Indeed, in \cite{AltPhillips} it is proved that if $u$ is a nonnegative minimizer of $E_\gamma^\mathrm{AP} [\cdot]$, then $u^{1/\beta}$ is $C^{1, \theta}$ for some $\theta>0$ inside $\{u>0\}\cap \Omega$ and up to $\partial\{u>0\}\cap \Omega$ in a neighborhood of free boundary regular points.
In the recent paper \cite{RestrepoRosOton} this has been improved, showing that $u^{1/\beta}$ is indeed $C^\infty$ up to the free boundary near regular points (and that this holds for solutions, not only minimizers). 
The same conclusion holds for $u/\dfb^{\beta}$.
These facts, in particular, yield that near regular free boundary points we have that nonnegative solutions satisfy \eqref{Eq:BehaviorNearFbu},
\begin{equation}
	\label{Eq:BehaviorNearFbNablau}
	|\nabla u| \asymp \dfb^{\beta - 1} = \dfb^{\frac{\gamma}{2 - \gamma}}
\end{equation}
and 
\begin{equation}
	\label{Eq:BehaviorNearFbHessianu}
	|D^2 u| \asymp \dfb^{\beta - 2} = \dfb^{2\frac{\gamma - 1}{2 - \gamma}}.
\end{equation}
Note that, despite the fact that in this article we always consider $\gamma \in (0,2)$, \eqref{Eq:BehaviorNearFbu} and \eqref{Eq:BehaviorNearFbNablau} are also true for $\gamma=0$ (recall that solutions to the one-phase problem are Lipschitz across the free boundary). 
Instead, the lower bound in \eqref{Eq:BehaviorNearFbHessianu} does not hold for $\gamma=0$,  as shown by the half-space solution $\max\{x_n, 0\}$.

\subsection{The modified Alt-Phillips functional}

The previous behavior of $u$ near the free boundary suggests studying the function $u^{1/\beta}$. 
Given $u\geq 0$, we define
\begin{equation}
	\label{Eq:Defv}
	v:= \beta u^{1/\beta}, \quad \text{ that is, } \quad  u = (v/\beta)^\beta.
\end{equation}
Then a simple computation noting that $1/\beta = 1 - \gamma/2$ shows that
\begin{equation}
	\label{Eq:RelationGradientsuv}
	\nabla v = \dfrac{\nabla u}{u^{\gamma/2}}, \quad \text{ that is, } \quad \nabla u = \left(\frac{v}{\beta}\right)^{\gamma \beta /2} \nabla v,
\end{equation}
and thus
\begin{equation}
	E_\gamma^\mathrm{AP} [u] = \left(\frac{1}{\beta}\right)^{\gamma \beta} \int_{\{v>0\}\cap \Omega} v^{\gamma \beta} (|\nabla v|^2 + 1) \d x.
\end{equation}
Hence, setting 
\begin{equation}
	\label{Eq:DefAlpha}
	\alpha := \gamma \beta = \dfrac{2\gamma}{2 - \gamma},
\end{equation}
it makes sense to define the modified functional
\begin{equation}
	\label{Eq:ModAltPhillipsFunctional}
	E_\alpha^\mathrm{M}[v] := \int_{\Omega} v^{\alpha} \chi_{\{v>0\}} \big(|\nabla v|^2 + 1\big) \d x.
\end{equation}

Nonnegative critical points of the functionals $E_\gamma^\mathrm{AP} [\cdot]$ and  $E_\alpha^\mathrm{M}[\cdot]$ are in correspondence through the relation \eqref{Eq:Defv}, and in particular minimality or stability properties are also shared.
Indeed, if we compute the expansion of $E_\gamma^\mathrm{AP}$ with respect to inner variations and up to second order (as in \Cref{Sec:GeneralExpansionAndStability}, see \eqref{Eq:ExpansionEnergyGeneral} below), setting $u = (v/\beta)^\beta$ we observe that the first and second order terms are, up to positive multiplicative constants, the same as in the expansion we would get for $E_\alpha^\mathrm{M}$. 
As a consequence, it follows that if $v$ is a nonnegative critical point of $E_\alpha^\mathrm{M}$, then 
\begin{equation}
	\label{Eq:ModAltPhillipsEquation}
	\Delta v = \dfrac{\alpha}{2} \dfrac{1 - |\nabla v|^2}{v} \quad \text{ in } \{v > 0\}\cap \Omega.
\end{equation}
This follows, using \eqref{Eq:RelationGradientsuv} and \eqref{Eq:AltPhillipsEquation}, from the computation
\begin{equation}
	\Delta v = u^{-\gamma/2} \Delta u - \dfrac{\gamma}{2} u^{-\gamma/2 - 1} |\nabla u|^2
	= \dfrac{\gamma}{2} \left(u^{\gamma/2 - 1} - u^{\gamma/2 - 1} |\nabla v|^2\right)
	= \dfrac{\gamma }{2} u^{-1/\beta } \left( 1 - |\nabla v|^2\right),
\end{equation}
using that $u^{-1/\beta } = \beta v^{-1}$ and recalling \eqref{Eq:DefAlpha}.
We will show in \Cref{Subsec:FirstVariation} how to obtain this equation directly from the first variation of $E_\alpha^\mathrm{M}$.
Moreover, if $\partial \{v > 0\}\cap \Omega$ is a smooth enough hypersurface, we also have the following free boundary condition:
\begin{equation}
	\label{Eq:ModAltPhillipsFBC}
	|\nabla v|^2 = 1 \quad \text{ in } \partial \{v > 0\}\cap \Omega.
\end{equation}
This will be obtained also from the first variation of $E_\alpha^\mathrm{M}$ in \Cref{Subsec:FirstVariation}.
Note that \eqref{Eq:ModAltPhillipsFBC} together with \eqref{Eq:RelationGradientsuv} gives a more precise expression of the the free boundary condition \eqref{Eq:AltPhillipsFBC} for $u$, since when $\partial \{u > 0\}\cap \Omega$ is a smooth enough hypersurface we have that
\begin{equation}
	\label{Eq:AltPhillipsFBCPrecise}
	\dfrac{|\nabla u|^2 }{u^\gamma}= 1 \quad \text{ in } \partial \{u > 0\}\cap \Omega.
\end{equation}
Both \eqref{Eq:ModAltPhillipsFBC} and \eqref{Eq:AltPhillipsFBCPrecise} must be understood as non-tangential limits as we approach the free boundary from the positive phase.

We conclude the section settling some notation.
Since along the paper we will usually integrate quantities in $\{u>0\} = \{v>0\}$ and sometimes we will use integration by parts, we define
\begin{equation}
	\label{Eq:DefNormal}
	\nu := -\dfrac{\nabla u}{|\nabla u|} = -\dfrac{\nabla v}{|\nabla v|},
\end{equation}
since $\nu$ is the outer unit normal to the positivity set $\{u>0\}$.
The surface measure on $\partial\Omega$ will be denoted by $\sigma$.

\section{A general stability condition with respect to domain variations}
\label{Sec:GeneralExpansionAndStability}

In this section we obtain a stability condition for a general functional
\begin{equation}
	\label{Eq:EnergyGeneral}
	E[u] := \int_\Omega G(u) (|\nabla u|^2 + F(u) ) \d x.
\end{equation}
Here, $F$ and $G$ are measurable functions for which we will not make any assumptions, as we want to work in full generality. 
So, for any particular choice of $F$ and $G$, we will consider $E[\cdot]$ evaluated in a set of functions for which \eqref{Eq:EnergyGeneral} is well defined.
Later in this paper, we will make a precise choice of $F$ and $G$. 
In particular, we would like to consider two cases:
\begin{itemize}
	\item $G (u) =  1$ and $F(u) = u^\gamma \chi_{\{u>0\}}$ with $\gamma \in [0,2)$.
	This corresponds to the Alt-Phillips functional \eqref{Eq:AltPhillipsFunctional}, including the Alt-Caffarelli problem (with $\gamma = 0$ we have $F(u) =  \chi_{\{u>0\}}$).
	For nonnegative $u$, this can be reformulated to $G(u) = \chi_{\{u>0\}}$ and $F(u) = u^\gamma$ with $\gamma \in [0,2)$, which in some instances may be more suitable to work with (for $\gamma = 0$ we have $F(u) = 1$).
	
	\item $G(v) =  v^\alpha \chi_{\{v>0\}}$ and $F(v) = 1$ with $\alpha >0$.
	This corresponds to the modified Alt-Phillips functional \eqref{Eq:ModAltPhillipsFunctional}, obtained  after the change $v = \beta u^{1/\beta}$, with $\alpha:= \gamma \beta$.
\end{itemize}
As will be mentioned later in \Cref{Remark:ExpansionForu}, the first case seems to lead to some problems for $\gamma >0$ (if $\gamma = 0$ both functionals are the same for nonnegative functions), and thus we will focus in the next section on the second one.

For a general functional as \eqref{Eq:EnergyGeneral} with $G\equiv 1$ and $F$ being a $C^1$ function, the usual way to obtain the Euler-Lagrange equation and the stability condition is to consider a competitor of the form $u + \varepsilon\varphi$ (sometimes called \textit{outer variation}) and simply differentiate with respect to $\varepsilon$, obtaining the semilinear equation $-\Delta u = f(u)$ in $\Omega$, where $f(u) := -F'(u)/2$, and the stability condition
\begin{equation}
	\label{Eq:StabilitySemilinear}
	- \int_{\Omega} \dfrac{F''(u)}{2} \varphi^2 \d x 
	= \int_{\Omega} f'(u) \varphi^2 \d x  
	\leq 
	\int_{\Omega} |\nabla \varphi|^2 \d x.
\end{equation}
When $F$ is not differentiable, however, one must consider \textit{inner variations}, i.e., performing a \textit{domain variation} taking competitors of the form $u\circ T_\varepsilon^{-1}$, with $T_\varepsilon$ a vector field which is the identity outside a compact set inside $\Omega$.
This is what we do in this section for the general functional \eqref{Eq:EnergyGeneral}, and will be used later with~ $E_\alpha^\mathrm{M}[\cdot]$.

The second variation formulas have a long history of applications in the theory of free boundary problems, in particular to study the vortex sheets \cite{GarabedianSchiffer}.
The latter paper contains the rigorous mathematical discussion of the classical Hadamard's variational formulas in space, with detailed calculations for the so-called interior variational method in three dimensions.
Note that the formula (3.2.22) in \cite{GarabedianSchiffer} is exactly \eqref{Eq:StabilityOnePhase-Intro}. 
The method can be easily generalized to higher dimensions; see \cite{BandleWagner} in the context of domain optimization problems. 
Observe that our second variation formula cannot be directly deduced from the one in 
\cite[Section~3]{BandleWagner}, because we do not assume that the energy density is $C^2$ with respect to $u$.
Here we compute the second variation for a functional of the form \eqref{Eq:EnergyGeneral} just using basic differential calculus, without employing any geometric notation or advanced tools.

\begin{remark}
	\label{Remark:ExteriorVariation}
	Using the approach of semilinear equations mentioned above, from \eqref{Eq:StabilitySemilinear} one can formally guess the stability inequality \eqref{Eq:StabilityOnePhase-Intro} in the Alt-Caffarelli problem ($\gamma = 0$).
	Indeed, taking $F(u) =  \chi_{\{u>0\}}$ then the distributional gradient of $F(u)$ is $ \nabla u /|\nabla u| \mathcal{H}^{n-1}\llcorner\partial\{u>0\} $, where $\mathcal{H}^{n-1}\llcorner\partial\{u>0\} $ is the restriction of the $n-1$ dimensional Hausdorff measure to $\partial{\{u>0\}}$. 
	Thus, the first and second variation of the term $|{\{u>0\}}\cap \Omega|$ in $E[\cdot]$ involve, respectively, the perimeter of ${\{u>0\}}$ and the mean curvature of $\partial{\{u>0\}}$.
	Thus, at least at a formal level, $f'(u) = H\, \mathcal{H}^{n-1}\llcorner\partial\{u>0\} $ in \eqref{Eq:StabilitySemilinear}. 
	However, when $F(u) = u^\gamma \chi_{\{u>0\}}$ with $\gamma >0$, this formal argument cannot be carried out, since the derivatives of $F(u)$ will involve $\mathcal{H}^{n-1}\llcorner\partial\{u>0\} $ and powers of $u$, and one should carefully analyze if there is a cancellation between singular and vanishing terms.
	This does not seem something intuitive to see at the formal level, as indicated by the nontrivial expression obtained in the right-hand side of our stability condition \eqref{Eq:StabilityAP1-intro} ---see \Cref{Th:StabilityCondition} below.
\end{remark}
 
We present next the expansion of \eqref{Eq:EnergyGeneral} when performing an inner variation.
In the statement and in the rest of the paper, for a mapping $\Phi = (\Phi^1, \ldots , \Phi^n)^\intercal$ we will denote its differential by 
\begin{equation}
	D\Phi = \left (
	\begin{matrix}
		\Phi^1_1 & \dots & \Phi^1_n \\ 
		\vdots &  & \vdots \\
		\Phi^n_1 & \dots & \Phi^n_n \\  
	\end{matrix} \right),
\end{equation}
where we use a subindex to denote a partial derivative, that is, $w_i := \partial_{x_i} w $ for any function $w$.
Moreover, here and through the paper we will adopt the usual convention of summation over repeated indexes.

\begin{proposition}
	\label{Prop:ExpansionAndStabilityGeneral}
	Let $\Omega \subset \R^n$ be a bounded domain with locally Lipschitz boundary, and let $u$ be an admissible function for the functional $E[\cdot]$ defined by \eqref{Eq:EnergyGeneral}.
	Given $\Phi \in C^\infty (\Omega;\R^n)$ with compact support in $\Omega$, for $\varepsilon>0$ let
	$$
	u_{\varepsilon }\left( x\right) := u(T_\varepsilon^{-1} (x) ), \quad \text{ with } \quad T_\varepsilon (x) :=x+\varepsilon \Phi(x).
	$$
	
	Then, for $\varepsilon$ small enough $T_\varepsilon$ is a diffeomorphism and thus $u_\varepsilon$ is well defined and an admissible competitor for $u$.
	Moreover, the following expansion holds:
	\begin{equation}
		\label{Eq:ExpansionEnergyGeneral}
		\begin{split}
			E[u_\varepsilon] & = \int_\Omega G(u) (|\nabla u|^2 + F(u) ) \d x \\
			& + \varepsilon \left ( \int_\Omega G(u) (|\nabla u|^2 + F(u) ) \div \Phi \d x -  2 \int_\Omega G(u)  \nabla u \cdot D\Phi \nabla u  \d x \right ) \\
			& + \varepsilon^2 \left ( \int_\Omega G(u) (|\nabla u|^2 + F(u) ) \dfrac{(\div\Phi)^2 - \Tr (D\Phi)^2}{2} \d x  
			+ \int_\Omega  G(u)  | D\Phi^\intercal \nabla u|^2 \d x  \right.\\
			& \quad \quad \quad \quad \left. + \int_\Omega G(u)  \big ( 2 u_i  u_j\Phi^i_l \Phi^l_j  - 2 \nabla u \cdot D\Phi \nabla u \div \Phi \big )\d x \right) + O(\varepsilon^3).
		\end{split}
	\end{equation}
	
	As a consequence, we say that $u$ is a stable critical point of $E[\cdot]$ if
	\begin{equation}
		\begin{split}
			\int_\Omega &G(u) (|\nabla u|^2 + F(u) ) \dfrac{(\div\Phi)^2 - \Tr (D\Phi)^2}{2} \d x  
			+ \int_\Omega  G(u)  | D\Phi^\intercal \nabla u|^2 \d x  \\
			& \quad \quad  + \int_\Omega G(u)  \big (2 u_i  u_j\Phi^i_l \Phi^l_j  - 2 \nabla u \cdot D\Phi \nabla u \div \Phi \big )\d x  \geq 0
		\end{split}		
	\end{equation}
	for every $\Phi \in C^\infty (\Omega;\R^n)$ with compact support in $\Omega$.	
\end{proposition}

\begin{proof}
	Note that since $\Phi$ has compact support, if $\varepsilon$ small enough, then $T_\varepsilon$ is a diffeomorphism and  $u_\varepsilon$ agrees with $u$ at $\partial \Omega$ ---thus it is an admissible competitor.
	Let us now compute the expansion~\eqref{Eq:ExpansionEnergyGeneral}.

	Using the change of variables $y = T_\varepsilon(x) $ we get
	\begin{equation}
		\label{Eq:ExpansionEnergyProof}
		\begin{split}
			E[u_\varepsilon] &= \int_\Omega G(u_\varepsilon) (|\nabla u_\varepsilon|^2 + F(u_\varepsilon) ) \d y \\
			&= \int_\Omega G(u(T_\varepsilon^{-1} (y))) (|\nabla u_\varepsilon (y) |^2 + F(u((T_\varepsilon^{-1} (y))) ) \d y\\
			& = \int_\Omega G(u(x)) \big( |\nabla u_\varepsilon (T_\varepsilon(x)) |^2 + F(u(x)) \big ) |\det(Id+\varepsilon D\Phi)(x)| \d x.
		\end{split}
	\end{equation}
	Note that we have the following expansion for the Jacobian:
	\begin{equation}
		\label{Eq:JacobianExpansion}
		\det(Id+\varepsilon D\Phi)=1+\varepsilon\div\Phi+\varepsilon^2 \dfrac{(\div\Phi)^2 - \Tr (D\Phi)^2}{2}+O(\varepsilon^3).
	\end{equation}
	For a proof of this just take $A=D\Phi$ in \Cref{Lemma:ExpansionDet}.
	In addition, we claim that
	\begin{equation}
		\label{Eq:ExpansionGradient}
		\begin{split}
			|\nabla u_\varepsilon(T_\varepsilon (x))|^2 &=  |\nabla u|^2 - 2 \varepsilon \nabla u \cdot D\Phi \nabla u \\
			& \quad \quad  + \varepsilon^2  \big (| D\Phi^\intercal \nabla u|^2 + 2 u_i  u_j\Phi^i_l \Phi^l_j \big )  + O(\varepsilon^3),
		\end{split}
	\end{equation}
	where everything in the right-hand side is evaluated at $x$.
	With this last expansion, the result will follow from using \eqref{Eq:JacobianExpansion} and  \eqref{Eq:ExpansionGradient} in \eqref{Eq:ExpansionEnergyProof}.

	It remains to show \eqref{Eq:ExpansionGradient}. 
	For $y\in \Omega$ we have
	\begin{equation}
		\label{Eq:ExpGradProof}
		|\nabla u_\varepsilon(y)|^2 = |\nabla (u( T_\varepsilon^{-1}(y))|^2 = |(D T_\varepsilon^{-1})^\intercal(y) \nabla u( T_\varepsilon^{-1}(y))|^2.
	\end{equation}
	Let us now obtain an expansion for $DT_\varepsilon^{-1}(y)$.
	Note that since $T_0 = Id$, we have a general expansion of the form
	$T_\varepsilon^{-1}(y) = y+\varepsilon a(y)+\varepsilon^2 b(y)+O(\varepsilon^3)$ for some $a,b\in C^\infty (\Omega;\Omega)$.
	But, since $y = T_\varepsilon( T_\varepsilon^{-1} (y) )$ we have
	\begin{equation}
		\begin{split}
			y &=  T_\varepsilon (y+\varepsilon a(y)+\varepsilon^2 b(y)+O(\varepsilon^3)) \\
			&= y+\varepsilon a(y)+\varepsilon^2 b(y)+O(\varepsilon^3) + \varepsilon \Phi\big (y+\varepsilon a(y)+\varepsilon^2 b(y)+O(\varepsilon^3)\big) \\
			&= y+\varepsilon a(y)+\varepsilon^2 b(y) + \varepsilon \big (\Phi(y) +\varepsilon D\Phi(y) a(y)+O(\varepsilon^2)\big) +O(\varepsilon^3) \\
			&= y+\varepsilon (a(y)+\Phi(y)) + \varepsilon^2 (b(y) + D\Phi(y) a(y) )  +O(\varepsilon^3).
		\end{split}
	\end{equation}
	Thus $a = -\Phi$ and $b = -D\Phi a = D\Phi \Phi$, that is,
	\begin{equation}
		\label{Eq:ExpansionInverse}
		T_\varepsilon^{-1}(y) = y-\varepsilon \Phi(y) + \varepsilon^2 (D\Phi \Phi)(y) + O(\varepsilon^3).
	\end{equation}
	Taking differentials we obtain
	\begin{equation}
		\label{Eq:ExpansionDifferentialInverse}
		DT_\varepsilon^{-1}(y) = Id-\varepsilon D\Phi(y) + \varepsilon^2 D(D\Phi \Phi)(y) + O(\varepsilon^3).
	\end{equation}

	Once we have the expansion \eqref{Eq:ExpansionDifferentialInverse}, we use \Cref{Lemma:ExpansionNormSqMatrixVector} with $q = \nabla u( T_\varepsilon^{-1}(y))$ and $M_\varepsilon = DT_\varepsilon^{-1}(y)$ (thus $A = - D\Phi(y)$ and $B = D(D\Phi \Phi)(y)$ in the notation of that lemma), and from \eqref{Eq:ExpGradProof} we get
	\begin{equation}
		\begin{split}
			|\nabla u_\varepsilon(y)|^2 
			&=  |\nabla u( T_\varepsilon^{-1}(y))|^2 - 2 \varepsilon \nabla u( T_\varepsilon^{-1}(y)) \cdot D\Phi(y) \nabla u( T_\varepsilon^{-1}(y)) \\
			& \quad \quad + \varepsilon^2  \big( | D\Phi^\intercal(y) \nabla u( T_\varepsilon^{-1}(y))|^2 + 2 \nabla u( T_\varepsilon^{-1}(y)) \cdot D(D\Phi \Phi ) (y)  \nabla u( T_\varepsilon^{-1}(y)) \big)\\
			& \quad \quad  + O(\varepsilon^3).
		\end{split}
	\end{equation}
	Hence, if we set $y = T_\varepsilon (x) = x + \varepsilon \Phi(x)$, we obtain
	\begin{equation}
		\begin{split}
			|\nabla u_\varepsilon(T_\varepsilon (x))|^2 
			&=  |\nabla u( x)|^2 
			- 2 \varepsilon \nabla u(x) \cdot D\Phi(x + \varepsilon\Phi(x)) \nabla u(x) \\
			& \quad \quad  
			+ \varepsilon^2  \big( | D\Phi^\intercal(x + \varepsilon\Phi(x)) \nabla u(x)|^2 + 2 \nabla u(x) \cdot D(D\Phi \Phi ) (x + \varepsilon\Phi(x))  \nabla u(x) \big)\\
			& \quad \quad  + O(\varepsilon^3).
		\end{split}
	\end{equation}
	Now, on the one hand, note that in the terms with $\varepsilon^2$ we can replace $x + \varepsilon\Phi(x)$ directly by $x$ since the rest only contributes to higher order terms.
	On the other hand, we will take a further expansion in the term $\nabla u(x) \cdot D\Phi(x + \varepsilon\Phi(x)) \nabla u(x)$ to obtain
	\begin{equation}
		\begin{split}
			\nabla u(x) \cdot D\Phi(x + \varepsilon\Phi(x)) \nabla u(x) 
			&= u_i (x) u_j(x) \Phi^i_j (x + \varepsilon\Phi(x)) \\
			&=  u_i (x) u_j(x) \Phi^i_j (x) + \varepsilon  u_i (x) u_j(x) \Phi^i_{jl} (x)  \Phi^l(x) +  O(\varepsilon^2).
		\end{split}
	\end{equation}
	Therefore, writing shortly $\nabla u \cdot D(D\Phi \Phi ) \nabla u = u_i (\Phi^i_l \Phi^l)_j u_j$ we finally obtain the expansion
	\begin{equation}
		\begin{split}
			|\nabla u_\varepsilon(T_\varepsilon (x))|^2 &=  |\nabla u|^2 - 2 \varepsilon \nabla u \cdot D\Phi \nabla u   + \varepsilon^2  \big (| D\Phi^\intercal \nabla u|^2 - 2  u_i  u_j \Phi^i_{jl}   \Phi^l + 2 u_i  u_j(\Phi^i_l \Phi^l)_j \big )  + O(\varepsilon^3),
		\end{split}
	\end{equation}
	where everything in the right-hand side is evaluated at $x$.
	This can be simplified to get \eqref{Eq:ExpansionGradient}, concluding the proof.
\end{proof}

\section{The stability condition for the Alt-Phillips functional}
\label{Sec:StabilityInequalityn-2}

In this section we establish one of the main results of this paper, which is the stability inequality for the Alt-Phillips problem.
The precise statement is the following:

\begin{theorem}
	\label{Th:StabilityCondition}
	Let $u\geq 0$ be a stable critical point of $E_\gamma^\mathrm{AP} [\cdot]$ or, equivalently, let $v = \beta u^{1/\beta}$ be a stable critical point of $E_\alpha^\mathrm{M} [\cdot]$.
	Assume that $v\in C^\infty(\overline{\{v>0\}}\cap \Omega)$, that the free boundary $\partial\{u>0\}$ is a smooth hypersurface in $\Omega$, and that $|\nabla v|>0$ in $\{v>0\}\cap \Omega$.
	
	Then,
	\begin{equation}
		\label{Eq:StabilityAP1}
		\dfrac{2-\gamma}{2} \dfrac{\gamma}{2}  
		\int_{\{u>0\} \cap \Omega} u^{\gamma} \dfrac{u^\gamma - |\nabla u|^2}{u^2}\varphi^2  \d x 
		\leq \int_{\{u>0\} \cap \Omega} u^{\gamma} |\nabla \varphi |^2  \d x
	\end{equation}
	or, equivalently,
	\begin{equation}
		\label{Eq:StabilityModifiedAltPhillips1}
		\dfrac{\alpha}{2} \int_{\{v>0\} \cap \Omega} v^\alpha \dfrac{1 - |\nabla v|^2}{v^2}  \varphi^2  \d x \leq \int_{\{v>0\} \cap \Omega} v^\alpha |\nabla \varphi |^2  \d x,
	\end{equation}
	for every $\varphi \in C^1_c(\Omega)$.
\end{theorem}

Before proceeding, a couple of remarks are in order.

\begin{remark}
	\label{Remark:AssumptionsStabilityCondition}
	Note that the assumption on the regularity of $v$ up to the free boundary is natural in view of the comments in \Cref{Sec:PreliminaryResults}, taking into account that we also assume that the free boundary is smooth.
	Regarding this second assumption, it is the usual one in the context of free boundary problems, and one of the technical points in the proofs (in \cite{CaffarelliJerisonKenig}, or in \Cref{Prop:Stabilityutau} below) is to use a cut-off function near singular free boundary points (typically the vertex of a cone) to circumvent this issue.	
	Finally, the nonvanishing assumption on $|\nabla v|$ is just a technical matter since in the proof we use a vector field of the form $\nabla v \varphi/|\nabla v|^2$ and we need to ensure that it is not singular. 
	Nevertheless, this assumption is not a problem in practical applications: the stability inequality is usually considered either for homogeneous solutions (for which the gradient does not vanish anywhere in the positivity set) or tested with $\varphi$ such that $|\varphi|\leq |\nabla v|\eta$ for a smooth function $\eta$ (and thus the nonvanishing assumption on $|\nabla v|$ can be removed).	
\end{remark}

\begin{remark}
	\label{Remark:IntegrabilityStabilityCondition}
	Recall that at regular free boundary points we have $u^\gamma = |\nabla u|^2 = 0$.
	The free boundary condition \eqref{Eq:ModAltPhillipsFBC} together with the relation \eqref{Eq:RelationGradientsuv} and the behavior of $u$ near the free boundary yields that 
	\begin{equation}
		\label{Eq:CancellationugammaMinusGradient}
		u^\gamma - |\nabla u|^2 \asymp \dfb^{\alpha + 1} = \dfb^{\frac{2 + \gamma}{2 - \gamma}}.
	\end{equation}
	Thus, the integrand in the left-hand side of \eqref{Eq:StabilityAP1} behaves as $\dfb^{\alpha - 1} \varphi^2$ near the free boundary, and thus it is integrable for all $\alpha >0$.
	Here we have used crucially the cancellation given by $u^\gamma - |\nabla u|^2$.
	This means that, if we split this difference into two terms in some computations, we only get integrands which behave as $\dfb^{\alpha - 2}$ near the free boundary, which would require $\alpha > 1$ (that is, $\gamma > 2/3$) so that the integrals are finite.
	The same happens if we split $1-|\nabla v|^2$ in \eqref{Eq:StabilityModifiedAltPhillips1}.		
\end{remark}

To establish the \Cref{Th:StabilityCondition} we use the general expansion \eqref{Eq:ExpansionEnergyGeneral} of \Cref{Prop:ExpansionAndStabilityGeneral} with $u$ replaced by $v$, and with $G(v) = v^\alpha \chi_{\{v>0\}}$ and $F(v) = 1$.
We have the following expansion:
\begin{equation}
	\label{Eq:ExpansionModifiedAltPhillips}
	\begin{split}
		E_\alpha^\mathrm{M}[v_\varepsilon] & = E_\alpha^\mathrm{M}[v] 
		+ \varepsilon  \int_{\{v>0\} \cap \Omega} \left( v^\alpha (|\nabla v|^2 + 1 ) \div \Phi - 2 \nabla v \cdot D\Phi \nabla v \right) \d x  \\
		&  \quad  + \varepsilon^2 \int_{\{v>0\} \cap \Omega} v^\alpha \left( (|\nabla v|^2 +1 ) \dfrac{(\div\Phi)^2 - \Tr (D\Phi)^2}{2} + |D\Phi^\intercal \nabla v|^2  + 2 v_i v_j(\Phi^i_l \Phi^l_j - \Phi^i_j \Phi^l_l)    \right) \d x  \\
		& \quad  + O(\varepsilon^3),
	\end{split}
\end{equation}
for any $\Phi \in C^1(\Omega;\R^n)$ with compact support in $\Omega$.
From this, the stability inequality \eqref{Eq:StabilityModifiedAltPhillips1} will be obtained by considering that the term with $\varepsilon^2$ is nonnegative, after a suitable choice of the vector field $\Phi$ in the direction of $\nabla v$.
For the sake of completeness, before the proof we will show next how to get, from the previous expansion, the equation \eqref{Eq:ModAltPhillipsEquation} and the free boundary condition~\eqref{Eq:ModAltPhillipsFBC} for critical points of the functional.

\subsection{First variation}
\label{Subsec:FirstVariation}

Assume that $v$ is a critical point of the functional $E_\alpha^\mathrm{M}[\cdot]$ and that $v$ is regular enough to carry out the computations below.
In particular, assume that $\partial \{v>0\} \cap \Omega$ is a smooth hypersurface.
Then,
\begin{equation}
	\int_{\{v>0\} \cap \Omega} \left( v^\alpha (|\nabla v|^2 + 1 ) \div \Phi - 2 \nabla v \cdot D\Phi \nabla v \right) \d x = 0
\end{equation}
for every $\Phi \in C^1(\Omega;\R^n)$ with compact support.

For $\varepsilon>0$, we take 
\begin{equation}
	\Phi = 
	\begin{cases}
		 v^{-\alpha} \Psi & \text{ in } \{v>\varepsilon\} \cap \Omega, \\
		 \varepsilon^{-\alpha} \Psi & \text{ in } \{0 <v \leq \varepsilon\} \cap \Omega, \\
	\end{cases}
\end{equation}
with $\Psi \in C^1(\Omega;\R^n)$ with compact support.
Note that since we are always integrating in $\{v>0\} \cap \Omega$ there is no need to specify the values of $\Phi$ in $\{v=0\} \cap \Omega$. 
As long as $\Phi$ is $C^1$ up to $\partial \{v>0\} \cap \Omega$ and the intersection of the free boundary with the support of $\Phi$ is smooth, we can extend $\Phi$ across the free boundary as a $C^1$ function with compact support.
Note also that, by approximation, we can take $\Phi$ just being Lipschitz and therefore the previous choice is admissible.

Taking $\Phi$ as above, we have
\begin{equation}
\begin{split}
	0 &= \int_{\{0 < v< \varepsilon \} \cap \Omega} \left(\dfrac{v}{\varepsilon} \right)^\alpha \left( (|\nabla v|^2 + 1 )  \Psi_j^j - 2v_i v_j \Psi_i^j  \right) \d x \\
	& \quad \quad + \int_{\{v>\varepsilon\} \cap \Omega} v^\alpha \left( (|\nabla v|^2 + 1 )  (v^{-\alpha} \Psi^j)_j - 2v_i v_j (v^{-\alpha} \Psi^j)_i  \right) \d x \\
	& = \int_{\{0 < v< \varepsilon \} \cap \Omega} \left(\dfrac{v}{\varepsilon} \right)^\alpha \left( (|\nabla v|^2 + 1 )  \Psi_j^j - 2v_i v_j \Psi_i^j  \right) \d x \\
	& \quad \quad - \alpha  \int_{\{v>\varepsilon\} \cap \Omega}  \left( (|\nabla v|^2 + 1 )  v^{-1} v_j \Psi^j - 2v_i^2 v_j v^{-1} \Psi^j \right) \d x \\
	& \quad \quad + \int_{\{v>\varepsilon\} \cap \Omega}  \left( (|\nabla v|^2 + 1 )  \Psi^j_j - 2v_i v_j \Psi^j_i  \right) \d x \\
	& = \int_{\{0 < v< \varepsilon \} \cap \Omega} \left(\dfrac{v}{\varepsilon} \right)^\alpha \left( (|\nabla v|^2 + 1 )  \Psi_j^j - 2v_i v_j \Psi_i^j  \right) \d x \\
	& \quad \quad - \alpha  \int_{\{v>\varepsilon\} \cap \Omega}   \dfrac{1 - |\nabla v|^2 }{v}  v_j \Psi^j  \d x 
	+ \int_{\{v>\varepsilon\} \cap \Omega}  \left( (|\nabla v|^2 + 1 )  \Psi^j_j - 2v_i v_j \Psi^j_i  \right) \d x.
\end{split}
\end{equation}
Now we let $\varepsilon \to 0$ and by dominated convergence the first integral converges to zero.
For the remaining terms, using integration by parts we obtain 
\begin{equation}
	\begin{split}
		0 &= - \alpha  \int_{\{v>0\} \cap \Omega}   \dfrac{1 - |\nabla v|^2 }{v}  v_j \Psi^j  \d x 
		+ \int_{\{v>0\} \cap \Omega}  \left( (|\nabla v|^2 + 1 )  \Psi^j_j - 2v_i v_j \Psi^j_i  \right) \d x \\
		&= - \alpha  \int_{\{v>0\} \cap \Omega}   \dfrac{1 - |\nabla v|^2 }{v}  v_j \Psi^j  \d x 
		- \int_{\partial \{v>0\} \cap \Omega}  \left( (|\nabla v|^2 + 1 ) \dfrac{v_j}{|\nabla v|} \Psi^j  - 2 \dfrac{v_i^2}{|\nabla v|} v_j\Psi^j  \right) \d \sigma \\
		& \quad \quad - \int_{\{v>0\} \cap \Omega}  \left( 2v_i v_{ij}  \Psi^j- 2v_{ii} v_j \Psi^j - 2v_i v_{ji} \Psi^j  \right) \d x \\
		& = - \alpha  \int_{\{v>0\} \cap \Omega}   \dfrac{1 - |\nabla v|^2 }{v}  v_j \Psi^j  \d x + \int_{\{v>0\} \cap \Omega}  2 \Delta v  v_j \Psi^j \d x  \\
		& \quad \quad - \int_{\partial \{v>0\} \cap \Omega} \dfrac{1 - |\nabla v|^2}{|\nabla v|}  v_j \Psi^j   \d \sigma .
	\end{split}
\end{equation}
From here we obtain both the equation \eqref{Eq:ModAltPhillipsEquation} and the free boundary condition \eqref{Eq:ModAltPhillipsFBC}.

\subsection{Second variation}

We now prove \Cref{Th:StabilityCondition}, obtaining the stability condition for the Alt-Phillips problem.

\begin{proof}[Proof of \Cref{Th:StabilityCondition}]
	We will establish first \eqref{Eq:StabilityModifiedAltPhillips1}.
	To do it, we abbreviate the second-order term in the energy expansion \eqref{Eq:ExpansionModifiedAltPhillips} as
	\begin{equation}
		\begin{split}
			\int_{\{v>0\} \cap \Omega} & v^\alpha \left( (|\nabla v|^2 +1 ) \dfrac{(\div\Phi)^2 - \Tr (D\Phi)^2}{2} + |D\Phi^\intercal \nabla v|^2  + 2 v_i v_j(\Phi^i_l \Phi^l_j - \Phi^i_j \Phi^l_l)    \right) \d x \\
			&=: I_1 + I_2 + I_3,
		\end{split}
	\end{equation} 
	and the stability inequality will follow from the fact that 
	$I_1 + I_2 + I_3 \geq 0$ after a suitable choice of $\Phi$ and some computations.
	As done in the first variation, since all the integrals are only computed in $\{v>0\}\cap \Omega$ and the free boundary is smooth, we will only specify the definition of our test functions in $\{v>0\}\cap \Omega$ assuming that they are extended to the whole $\Omega$ as a $C^1$ function with compact support.
	
	We take 
	\begin{equation}
		\Phi = \nabla v \xi,
	\end{equation}
	with $\xi \in C^1_c(\Omega)$ and deal with each term separately. 
	Eventually we will take 
	\begin{equation}
		\xi = \varphi/ |\nabla v|^2
	\end{equation}
	with $\varphi \in C^1_c(\Omega)$.
	Note that $v$ is $C^\infty$ up to the free boundary near regular points, and since $|\nabla v|$ does not vanish, this is an admissible test function to choose.

	Now we compute each term in the second variation separately,
	just using some general differential identities (in particular, without using any knowledge of which equation $v$ satisfies).
	Later we will add $I_1$, $I_2$, and $I_3$, and we will integrate by parts in some terms and finally use the equation to get~ \eqref{Eq:StabilityModifiedAltPhillips1}.
	
	\medskip
	$\bullet$ \textbf{Computation of $I_1$.}	
	On the one hand, 
	\begin{equation}
		(\div\Phi)^2 = (\div (\nabla v \xi) )^2 = (\Delta v \xi + \nabla v \cdot \nabla \xi)^2= (\Delta v)^2 \xi^2 +2 \Delta v \xi \nabla v \cdot \nabla \xi+ (\nabla v \cdot \nabla \xi)^2 .
	\end{equation}
	On the other hand, since for a matrix $A = (a_{ij})$ it holds $ \Tr A^2 = a_{i j } a_{ji}$, we have
	\begin{equation}
		\begin{split}
			\Tr (D\Phi)^2 &= \Phi^i_j \Phi^j_i = (v_i \xi)_j (v_j \xi)_i 
			= (v_{ij}\xi + v_i \xi_j)(v_{ji} \xi + v_j \xi_i) \\
			& = v_{ij}^2 \xi^2 + v_{ij} v_j \xi \xi_i +  v_i v_{ji} \xi \xi_j + v_i \xi_j v_j \xi_i \\
			& = v_{ij}^2 \xi^2 + 2 v_{ij} v_j \xi \xi_i + (\nabla v \cdot \xi)^2.
		\end{split}
	\end{equation}
	Thus, we get 
	\begin{equation}
		\begin{split}
			(\div\Phi)^2 - \Tr (D\Phi)^2 
			&= (\Delta v)^2 \xi^2 +  \Delta v \nabla v \cdot \nabla \xi^2    - v_{ij}^2 \xi^2 - \dfrac{1}{2} \nabla (|\nabla v|^2) \cdot \nabla \xi^2 \\
			& = \div \left[ \left(\Delta v \nabla v - \dfrac{1}{2} \nabla (|\nabla v|^2)\right ) \xi^2 \right] ,
		\end{split}	
	\end{equation}
	where in the last equality we have used that $\Delta  (|\nabla v|^2) ) = 2 v_{ij}^2 + 2 v_i v_{ijj} = 2 v_{ij}^2 + 2 \nabla (\Delta v) \cdot \nabla v$.
	Consequently, setting $\xi = \varphi/|\nabla v|^2$ we obtain
	\begin{equation}
		I_1 =  \dfrac{1}{2}\int_{\{v>0\} \cap \Omega} v^\alpha (|\nabla v|^2 +1 ) \div \left[ \left( \Delta v \dfrac{\nabla v }{|\nabla v|^4} - \dfrac{1}{2} \dfrac{\nabla (|\nabla v|^2)}{|\nabla v|^4} \right ) \varphi^2 \right] \d x.
	\end{equation}

	\medskip
	$\bullet$ \textbf{Computation of $I_2$.}
	Note that
	\begin{equation}
		|D\Phi^\intercal \nabla v|^2 = \sum_i \left[\sum_m v_{m} \Phi^{m}_{i} \right]^2.
	\end{equation}
	Thus, after the choice $\Phi = \nabla v \xi$, we get
	\begin{equation}
		\begin{split}
			|D\Phi^\intercal \nabla v|^2 & = \sum_i \left[\sum_m v_{m} (v_m \xi)_i \right]^2 = \sum_i \left[\sum_m v_{m}^2 \xi_i + \sum_m v_{m} v_{mi} \xi \right]^2 
			= \sum_i \left[|\nabla v|^2 \xi_i + \dfrac{1}{2} (|\nabla v|^2)_i \xi \right]^2.
		\end{split}
	\end{equation}
	Now, let us take $\xi = \varphi / |\nabla v|^2$. 
	We have
	\begin{equation}
		\begin{split}
			|\nabla v|^2 \xi_i + \dfrac{1}{2} (|\nabla v|^2)_i \xi & = |\nabla v|^2 \left(\frac{\varphi}{|\nabla v|^2}  \right)_i +  \dfrac{1}{2} \dfrac{(|\nabla v|^2)_i}{|\nabla v|^2} \varphi 
			= \varphi_i - |\nabla v|^2 \frac{1}{|\nabla v|^4}  (|\nabla v|^2)_i \varphi +  \dfrac{1}{2} \dfrac{(|\nabla v|^2)_i}{|\nabla v|^2} \varphi \\
			&=  \varphi_i - \dfrac{1}{2} \dfrac{(|\nabla v|^2)_i}{|\nabla v|^2} \varphi .
		\end{split}
	\end{equation}
	Hence, 
	\begin{equation}
		\begin{split}
			|D\Phi^\intercal \nabla v|^2 & =  \sum_i \left[\varphi_i - \dfrac{1}{2} \dfrac{(|\nabla v|^2)_i}{|\nabla v|^2} \varphi \right]^2 = \sum_i \left [ \varphi_i^2 -  \dfrac{(|\nabla v|^2)_i}{|\nabla v|^2} \varphi  \varphi_i + \dfrac{1}{4} \dfrac{(|\nabla v|^2)_i^2}{|\nabla v|^4} \varphi^2 \right] \\
			&= |\nabla \varphi |^2 - \dfrac{1}{2} \dfrac{\nabla(|\nabla v|^2)}{|\nabla v|^2} \cdot \nabla \varphi^2 + \dfrac{1}{4} \dfrac{|\nabla (|\nabla v|^2)|^2}{|\nabla v|^4} \varphi^2 \\
			&= |\nabla \varphi |^2 
			- \dfrac{1}{2} \div \left( \dfrac{\nabla(|\nabla v|^2)}{|\nabla v|^2}  \varphi^2  \right) 
			+ \dfrac{1}{2} \dfrac{\Delta(|\nabla v|^2)}{|\nabla v|^2}  \varphi^2 
			- \dfrac{1}{4} \dfrac{|\nabla (|\nabla v|^2)|^2}{|\nabla v|^4} \varphi^2.
		\end{split}
	\end{equation}

	For convenience in forthcoming computations, let us compute the Laplacian in the third term.
	We get
	\begin{equation}
		\dfrac{1}{2}  \dfrac{\Delta(|\nabla v|^2)}{|\nabla v|^2}  \varphi^2   
		=  \dfrac{\nabla (\Delta v) \cdot \nabla v}{|\nabla v|^2}  \varphi^2  
		+ \dfrac{v_{ij}^2}{|\nabla v|^2} \varphi^2 
	\end{equation}
	and thus we obtain 
	\begin{equation}
		\begin{split}
			I_2 &= \int_{\{v>0\} \cap \Omega} v^\alpha |\nabla \varphi |^2  \d x 
			- \dfrac{1}{2}\int_{\{v>0\} \cap \Omega} v^\alpha \div \left( \dfrac{\nabla(|\nabla v|^2)}{|\nabla v|^2}  \varphi^2  \right)   \d x \\
			& \quad \quad + \int_{\{v>0\} \cap \Omega} v^\alpha \dfrac{\nabla (\Delta v) \cdot \nabla v}{|\nabla v|^2}  \varphi^2   \d x + 
			\int_{\{v>0\} \cap \Omega} v^\alpha \dfrac{v_{ij}^2}{|\nabla v|^2} \varphi^2  \d x 
			- \dfrac{1}{4} 	\int_{\{v>0\} \cap \Omega} v^\alpha \dfrac{|\nabla (|\nabla v|^2)|^2}{|\nabla v|^4} \varphi^2 \d x.
		\end{split} 
	\end{equation}

	
	\medskip
	$\bullet$ \textbf{Computation of $I_3$.}	
	Taking $\Phi = \nabla v \xi$, we compute
	\begin{equation}
		\begin{split}
			\Phi^i_l \Phi^l_j-\Phi^i_j \Phi^l_l 
			&= (v_i \xi)_l (v_l \xi)_j-(v_i \xi)_j (v_l \xi)_l 
			= (v_{il} \xi + v_i \xi_l) (v_{lj}\xi + v_l \xi_j) - (v_{ij} \xi + v_i \xi_j) (v_{ll} \xi + v_l \xi_l) \\
			&= v_{il} v_{lj} \xi^2 + v_l v_{il} \xi \xi_j + v_i v_{lj} \xi \xi_l + v_i \xi_l v_l \xi_j 
			- v_{ij} v_{ll} \xi^2 - v_{ij}  v_l \xi\xi_l -  v_i v_{ll} \xi \xi_j - v_i \xi_j v_l \xi_l \\
			&= (v_{il} v_{lj} - v_{ij} \Delta v) \xi^2 
			+ \dfrac{1}{2} \left ( \dfrac{1}{2} (|\nabla v|^2)_i - v_i \Delta v \right) (\xi^2)_j 
			+ \dfrac{1}{2}( v_i v_{lj} -  v_{ij}  v_l)  (\xi^2)_l  . 
		\end{split}
	\end{equation}
	Hence, 
	\begin{equation}
		\begin{split}
			2 v_i v_j (\Phi^i_l \Phi^l_j-\Phi^i_j \Phi^l_l )
			&= 2 ( v_i v_{il}  v_j  v_{lj} - v_i v_j  v_{ij} \Delta v) \xi^2 
			+ \left ( \dfrac{1}{2}  v_i (|\nabla v|^2)_i - v_i^2 \Delta v \right) v_j (\xi^2)_j \\
			& \quad \quad + ( v_i^2 v_j v_{lj} - v_i v_j  v_{ij}  v_l)  (\xi^2)_l  \\
			&= \left ( \dfrac{1}{2} |\nabla (|\nabla v|^2)|^2 -  \nabla v \cdot \nabla (|\nabla v|^2) \Delta v\right ) \xi^2 \\
			&\quad  \quad + \left ( \dfrac{1}{2}  \nabla v \cdot \nabla (|\nabla v|^2) - |\nabla v|^2 \Delta v \right) \nabla v \cdot \nabla \xi^2 \\
			& \quad \quad + \left( \dfrac{1}{2} |\nabla v|^2 \nabla (|\nabla v|^2) -\dfrac{1}{2} \nabla v \cdot   (|\nabla v|^2) \nabla v \right) \cdot  \nabla \xi^2 \\
			&= \left ( \dfrac{1}{2} |\nabla (|\nabla v|^2)|^2 -  \nabla v \cdot \nabla (|\nabla v|^2) \Delta v\right ) \xi^2  \\
			& \quad \quad + \dfrac{1}{2} |\nabla v|^2 \nabla (|\nabla v|^2) \cdot  \nabla \xi^2  
			-  |\nabla v |^2  \Delta v \nabla v \cdot \nabla \xi^2\\
			& = \div \left[ \left( \dfrac{1}{2} |\nabla v|^2 \nabla (|\nabla v|^2) -  |\nabla v |^2  \Delta v \nabla v   \right)\xi^2 \right] \\
			& \quad \quad - \dfrac{1}{2}|\nabla v |^2 \Delta 	(|\nabla v |^2) \xi ^2 + |\nabla v |^2 \nabla (\Delta v) \cdot \nabla v \xi^2 + |\nabla v |^2 (\Delta v)^2 \xi^2 \\
			& = \div \left[ \left( \dfrac{1}{2} |\nabla v|^2 \nabla 	(|\nabla v|^2) -  |\nabla v |^2  \Delta v \nabla v   \right)\xi^2 \right]   + |\nabla v |^2 ((\Delta v)^2 - v_{ij}^2) \xi^2 .
		\end{split}
	\end{equation}
	In the last equality we have used again the identity $\Delta  (|\nabla v|^2) ) = 2 v_{ij}^2 + 2 v_i v_{ijj} = 2 v_{ij}^2 + 2 \nabla (\Delta v) \cdot \nabla v$.
	Taking $\xi = \varphi / |\nabla v|^2$, we have obtained
	\begin{equation}
		\begin{split}
			I_3 &= \int_{\{v>0\} \cap \Omega} v^\alpha \dfrac{(\Delta v)^2 - v_{ij}^2}{|\nabla v|^2} \varphi^2 \d x + \int_{\{v>0\} \cap \Omega} v^\alpha \div \left[ \left( \dfrac{1}{2} \dfrac{\nabla (|\nabla v|^2)}{|\nabla v |^2} -    \dfrac{\Delta v \nabla v }{|\nabla v |^2}  \right)\varphi^2 \right] \d x.
		\end{split} 
	\end{equation}
	
	\pagebreak
	
	\medskip
	$\bullet$ \textbf{Combining all three terms, integrating by parts, and using the equation.}
	
	Putting together the previous computations we obtain
	\begin{equation}
		\label{Eq:2ndVarModifiedAltPhillipsPreInt}
		\begin{split}
			I_1 + I_2 + I_3 
			& = \int_{\{v>0\} \cap \Omega} v^\alpha |\nabla \varphi |^2  \d x \\
			& \quad \quad + \dfrac{1}{2}\int_{\{v>0\} \cap \Omega} v^\alpha (|\nabla v|^2 +1 ) \div \left[ \left( \Delta v \dfrac{\nabla v }{|\nabla v|^4} - \dfrac{1}{2} \dfrac{\nabla (|\nabla v|^2)}{|\nabla v|^4} \right ) \varphi^2 \right] \d x \\
			& \quad \quad -\int_{\{v>0\} \cap \Omega} v^\alpha \div \left[   \dfrac{\Delta v \nabla v }{|\nabla v |^2} \varphi^2 \right] \d x \\
			& \quad \quad 
			+ \int_{\{v>0\} \cap \Omega} v^\alpha \dfrac{\nabla (\Delta v) \cdot \nabla v}{|\nabla v|^2}  \varphi^2   \d x 
			- \dfrac{1}{4} \int_{\{v>0\} \cap \Omega} v^\alpha \dfrac{|\nabla (|\nabla v|^2)|^2}{|\nabla v|^4} \varphi^2 \d x \\
			& \quad \quad + \int_{\{v>0\} \cap \Omega} v^\alpha \dfrac{(\Delta v)^2 }{|\nabla v|^2} \varphi^2 \d x. \\
		\end{split}
	\end{equation}
	We denote by $I_{\mathrm{parts}}$ the terms in the previous expression which we will integrate by parts, that is,
	\begin{equation}
		\label{Eq:Iparts}
		\begin{split}
			I_{\mathrm{parts}} &:= 
			\dfrac{1}{2}\int_{\{v>0\} \cap \Omega} v^\alpha (|\nabla v|^2 +1 ) \div \left[ \left( \Delta v \dfrac{\nabla v }{|\nabla v|^4} - \dfrac{1}{2} \dfrac{\nabla (|\nabla v|^2)}{|\nabla v|^4} \right ) \varphi^2 \right] \d x \\
			& \quad \quad  -\int_{\{v>0\} \cap \Omega} v^\alpha \div \left[   \dfrac{\Delta v \nabla v }{|\nabla v |^2} \varphi^2 \right] \d x.
		\end{split}
	\end{equation} 
	Integrating by parts we obtain 
	\begin{equation}
		\begin{split}
			I_{\mathrm{parts}}  & = - \dfrac{1}{2}\int_{\{v>0\} \cap \Omega} v^\alpha  (  \alpha v^{-1}  (|\nabla v|^2 +1 ) \nabla v  + \nabla (|\nabla v|^2 ) ) \cdot  \left( \Delta v \dfrac{\nabla v }{|\nabla v|^4} - \dfrac{1}{2} \dfrac{\nabla (|\nabla v|^2)}{|\nabla v|^4} \right ) \varphi^2  \d x \\
			& \quad \quad + \alpha \int_{\{v>0\} \cap \Omega} v^\alpha     \dfrac{\Delta v  }{v} \varphi^2  \d x \\
			&=   - \dfrac{\alpha}{2}\int_{\{v>0\} \cap \Omega} v^\alpha  (|\nabla v|^2 +1 )  \dfrac{\Delta v}{v |\nabla v|^2}   \varphi^2  \d x 
			+ \dfrac{\alpha}{4}\int_{\{v>0\} \cap \Omega} v^\alpha  \dfrac{|\nabla v|^2 +1}{v} \dfrac{\nabla v \cdot \nabla (|\nabla v|^2)}{|\nabla v|^4}  \varphi^2  \d x \\
			& \quad \quad - \dfrac{1}{2}\int_{\{v>0\} \cap \Omega} v^\alpha    \Delta v \dfrac{\nabla (|\nabla v|^2 )  \cdot \nabla v }{|\nabla v|^4} \varphi^2  \d x 
			+ \dfrac{1}{4}\int_{\{v>0\} \cap \Omega} v^\alpha  \dfrac{|\nabla (|\nabla v|^2)|^2}{|\nabla v|^4}  \varphi^2  \d x \\
			& \quad \quad + \alpha \int_{\{v>0\} \cap \Omega} v^\alpha     \dfrac{\Delta v  }{v} \varphi^2  \d x.
		\end{split}
	\end{equation}
	Note that since the free boundary is smooth, the factor $v^\alpha$ with $\alpha>0$ makes all boundary term vanish in the integration by parts.\footnote{\label{Footnote:BoundaryTerm}	
		This is a step where the case $\alpha>0$ is significantly different from $\alpha =  0$; see \Cref{Sec:AlphaToZero} for more details.
	} 
	Moreover, note that if $\alpha \in (0,1)$, to do the integration by parts correctly we should integrate first in $\{v > \delta\}$ for  $\delta >0$, where $v^\alpha$ is smooth, and then let $\delta \to 0$ to obtain the previous identity by dominated convergence.
	Combining the first and last terms in $I_{\mathrm{parts}}$, and also the second one with the third one, from the last expression we get		
	\begin{equation}
		\begin{split}		
			I_{\mathrm{parts}}  & = \alpha \int_{\{v>0\} \cap \Omega} v^\alpha     \dfrac{\Delta v  }{v |\nabla v|^2 }   \left( |\nabla v|^2 - \dfrac{1 + |\nabla v|^2}{2}\right)\varphi^2 \d x 
			+ \dfrac{1}{4}\int_{\{v>0\} \cap \Omega} v^\alpha  \dfrac{|\nabla (|\nabla v|^2)|^2}{|\nabla v|^4}  \varphi^2  \d x  \\ &\quad \quad + \dfrac{1}{2}\int_{\{v>0\} \cap \Omega} v^\alpha  \left( \dfrac{\alpha }{2} \dfrac{|\nabla v|^2 +1}{v}   - \Delta v \right) \dfrac{\nabla v \cdot \nabla (|\nabla v|^2)}{|\nabla v|^4}  \varphi^2  \d x  \\
			& =  - \int_{\{v>0\} \cap \Omega} v^\alpha  \dfrac{\Delta v  }{ |\nabla v|^2 }   \dfrac{ \alpha }{2}   \dfrac{1 -|\nabla v|^2}{v} \varphi^2 \d x 
			+ \dfrac{1}{4}\int_{\{v>0\} \cap \Omega} v^\alpha  \dfrac{|\nabla (|\nabla v|^2)|^2}{|\nabla v|^4}  \varphi^2  \d x  \\ &\quad \quad + \dfrac{1}{2}\int_{\{v>0\} \cap \Omega} v^\alpha  \left( \dfrac{\alpha }{2} \dfrac{|\nabla v|^2 +1}{v}   - \Delta v \right) \dfrac{\nabla v \cdot \nabla (|\nabla v|^2)}{|\nabla v|^4}  \varphi^2  \d x.
		\end{split}
	\end{equation}
	Now, using the equation for $v$ this gets simplified to 
	\begin{equation}
		\begin{split}
			I_{\mathrm{parts}}  
			& =  - \int_{\{v>0\} \cap \Omega} v^\alpha  \dfrac{(\Delta v)^2  }{|\nabla v|^2 }  \varphi^2 \d x 
			+ \dfrac{1}{4}\int_{\{v>0\} \cap \Omega} v^\alpha  \dfrac{|\nabla (|\nabla v|^2)|^2}{|\nabla v|^4}  \varphi^2  \d x  \\ &\quad \quad +  \dfrac{\alpha}{2}\int_{\{v>0\} \cap \Omega} v^\alpha  \dfrac{1}{v} \dfrac{\nabla v \cdot \nabla (|\nabla v|^2)}{|\nabla v|^2}  \varphi^2  \d x.
		\end{split}
	\end{equation}
	Plugging this into \eqref{Eq:2ndVarModifiedAltPhillipsPreInt} we get
	\begin{equation}
		\label{Eq:2ndVarModifiedAltPhillipsAfterInt}
		I_1 + I_2 + I_3  
		= \int_{\{v>0\} \cap \Omega} v^\alpha |\nabla \varphi |^2  \d x + \int_{\{v>0\} \cap \Omega} v^\alpha  \dfrac{\nabla v}{|\nabla v|^2} \cdot   \left( \dfrac{\alpha}{2}\dfrac{\nabla (|\nabla v|^2)}{v} + \nabla (\Delta v) \right) \varphi^2  \d x.
	\end{equation}
	Using again the equation, we see that in $\{v>0\} \cap \Omega$ it holds
	\begin{equation}
		\begin{split}
			\dfrac{\alpha}{2}\dfrac{\nabla (|\nabla v|^2)}{v} + \nabla (\Delta v) & = \dfrac{\alpha}{2} \left( \dfrac{\nabla (|\nabla v|^2)}{v} + \nabla \left [ \dfrac{1 - |\nabla v|^2}{v} \right ]  \right) = - \dfrac{\alpha}{2} \dfrac{1 - |\nabla v|^2}{v^2} \nabla v  .
		\end{split}
	\end{equation}
	Therefore,
	\begin{equation}
		I_1 + I_2 + I_3
		= \int_{\{v>0\} \cap \Omega} v^\alpha |\nabla \varphi |^2  \d x  
		- \dfrac{\alpha}{2} \int_{\{v>0\} \cap \Omega} v^\alpha \dfrac{1 - |\nabla v|^2}{v^2}  \varphi^2  \d x,
	\end{equation}
	and we obtain the stability condition \eqref{Eq:StabilityModifiedAltPhillips1} from the fact that $I_1 + I_2 + I_3 \geq 0$.

	\medskip
	$\bullet$ \textbf{Proof of \eqref{Eq:StabilityAP1}.}	
	To conclude the proof, let us show that the inequality \eqref{Eq:StabilityModifiedAltPhillips1} is equivalent to \eqref{Eq:StabilityAP1}.
	To do it, we undo the change \eqref{Eq:Defv}, that is, we use that $v = \beta u^{1/\beta}$.
	Indeed, we have that $v^\alpha = \beta^\alpha u^\gamma$ and thus the right-hand side of \eqref{Eq:StabilityModifiedAltPhillips1} can be written as
	\begin{equation}
		\int_{\{v>0\} \cap \Omega} v^\alpha |\nabla \varphi |^2  \d x
		= \beta^\alpha \int_{\{u>0\} \cap \Omega} u^\gamma |\nabla \varphi |^2  \d x ,
	\end{equation}
	while the left-hand side of \eqref{Eq:StabilityModifiedAltPhillips1} becomes
	\begin{equation}
			\dfrac{\alpha}{2} \int_{\{v>0\} \cap \Omega} v^\alpha \dfrac{1 - |\nabla v|^2}{v^2}  \varphi^2  \d x
			=
			\dfrac{\gamma \beta}{2}  \beta^\alpha \int_{\{u>0\} \cap \Omega}  u^\gamma \dfrac{1 - |\nabla v|^2}{v^2}\d x.
	\end{equation}
	Therefore, it only remains to show that
	\begin{equation}
		\beta \dfrac{1 - |\nabla v|^2}{v^2} = \dfrac{2 - \gamma}{2} \dfrac{u^\gamma - |\nabla u|^2}{u^2} \quad \text{ in } \{u>0\} \cap \Omega.
	\end{equation}
	To show this last identity, recall that from \eqref{Eq:RelationGradientsuv} we have $|\nabla v|^2 = |\nabla u|^2/u^\gamma$ in $\{u>0\} \cap \Omega$ and thus in this set we have
	\begin{equation}
		\beta \dfrac{1 - |\nabla v|^2}{v^2}
		= \beta \dfrac{1 - |\nabla u|^2/u^\gamma}{\beta^2 u^{2/\beta}}
		= \dfrac{1}{\beta} \dfrac{u^\gamma -  |\nabla u|^2}{u^\gamma u^{2-\gamma}} = \dfrac{2 - \gamma}{2} \dfrac{u^\gamma - |\nabla u|^2}{u^2}.
	\end{equation}
\end{proof}

\begin{remark}
	\label{Remark:ExpansionForu}
	One may think that the same computations could be carried out directly for the functional $E_\gamma^\mathrm{AP}[\cdot]$ instead of $E_\alpha^\mathrm{M}[\cdot]$.
	Indeed, our first approach to this problem was precisely that, using an expansion for $E_\gamma^\mathrm{AP}[\cdot]$ taking later a vector field of the form $\nabla u \varphi/|\nabla u|^2$ (since this is the choice that works when $\gamma = 0$, see \Cref{Sec:AlphaToZero}).
	Nevertheless, several difficulties appear in this case, since  $|\nabla u|$ vanishes on the free boundary and thus this choice is not admissible.
	Even trying with a vector field of the form $\nabla u \varphi/|\nabla u|$ gives problems, since the computations lead to singular free boundary terms in the integration by parts.
	Once we see how the previous proof works, this tells us that the right choice, if one wants to carry out the computations directly for $E_\gamma^\mathrm{AP}[\cdot]$, is taking a vector field of the form  $\nabla u \varphi u^{\gamma/2}/|\nabla u|^2$.
\end{remark}

\section{Recovering the stability condition of the Alt-Caffarelli problem}
\label{Sec:AlphaToZero}

In this section, we briefly comment on the relation of \Cref{Th:StabilityCondition} with the stability inequality for the Alt-Caffarelli  problem \eqref{Eq:StabilityOnePhase-Intro}, which would correspond to the case $\gamma = 0$ (i.e., $\alpha =0$).
The key fact that we will use is the following lemma: 

\begin{lemma}
	\label{Lemma:MeandCurvatureFB}
	Let $\alpha \geq 0$ and $v\in C^\infty(\overline{\{v>0\}}\cap \Omega)$ be a function satisfying \eqref{Eq:ModAltPhillipsEquation} and \eqref{Eq:ModAltPhillipsFBC}, and such that $\partial \{v>0\}\cap \Omega$ is smooth.
	Then, for every $x_0\in \partial \{v>0\}\cap \Omega$,
	\begin{equation}
		\label{Eq:LaplacianFreeBoundary}
		\lim_{x \to x_0}\Delta v(x) = - \alpha v_{\nu \nu}(x_0)
	\end{equation}
	whenever $x \to x_0$ from the positivity set $\{v>0\}$.
	In particular, the mean curvature\footnote{See \Cref{Footnote:MeanCurvature} for our orientation convention.} of the free boundary at regular points can be expressed as
	\begin{equation}
		\label{Eq:MeanCurvatureFreeBoundaryvnunu}
		H = - (1 + \alpha) v_{\nu\nu}.
	\end{equation}
\end{lemma}

\begin{proof}
	Given $x_0\in \partial \{v>0\}\cap \Omega$ a regular free boundary point, for $t >0$ small enough we consider the points $x_0 -t\nu(x_0) \in \{v>0\}$.
	For simplicity we will write $\nu = \nu(x_0)$.
	Using the PDE \eqref{Eq:ModAltPhillipsEquation} we get
	\begin{equation}
		\Delta v(x_0-t\nu) = \dfrac{\alpha}{2} (1 - |\nabla v(x_0 - t\nu)|^2) \dfrac{1}{v(x_0 - t\nu)}.
	\end{equation}
	We will expand the previous expression in terms of $t$ and later take the limit $t\downarrow0$.

	On the one hand, using that $v(x_0)=0$, we know that
	\begin{equation}
		v(x_0 - t\nu) = v(x_0) - t v_\nu (x_0) + O(t^2) = - t v_\nu (x_0) + O(t^2),
	\end{equation}
	and since $|\nabla v| = 1$ along the free boundary, it follows that $v_\nu (x_0) =-1$, which yields
	\begin{equation}
		\label{Eq:vExpNearFB1}
		v(x_0 - t\nu) = t  + O(t^2).
	\end{equation}
	On the other hand, since for $i=1,\ldots,n$ we have
	\begin{equation}
		v_i(x_0-t\nu) = v_i(x_0) - t v_{i \nu} (x_0) + O(t^2),
	\end{equation}
	it follows that
	\begin{equation}
		\label{Eq:GradSquaredExpNearFB1}
		\begin{split}
			|\nabla v(x_0 - t\nu)|^2 &= \sum_{i=1}^n v_i(x_0-t\nu)^2 = \sum_{i=1}^n  v_i(x_0)^2 - 2 t \sum_{i=1}^n  v_i(x_0) v_{i \nu} (x_0) + O(t^2)\\
			& = |\nabla v(x_0)|^2 - t \partial_\nu |\nabla v(x_0)|^2 + O(t^2),
		\end{split}
	\end{equation}
	Note that if we choose coordinates at $x_0$ is such a way that $e_n = \nu = -\nabla v/ |\nabla v|$, then $\nabla v (x_0) = -|\nabla v (x_0)|e_n$ and thus
	\begin{equation}
		v_i (x_0) = 0 \quad \text{ for } i=1,\ldots,n-1
	\end{equation}
	and 
	\begin{equation}
		v_n (x_0) = e_n \cdot \nabla v(x_0) = -\nabla v(x_0) \cdot \nabla v(x_0)/|\nabla v (x_0)| = -|\nabla v (x_0)|.
	\end{equation}
	As a consequence, since $v_\nu (x_0) = v_n(x_0)$ in these coordinates,
	\begin{equation}
		\partial_\nu |\nabla v(x_0)|^2 
		= 2\sum_{i=1}^n  v_i(x_0) v_{i \nu} (x_0) 
		= 2 v_n(x_0) v_{n \nu} (x_0) 
		= -2 |\nabla v (x_0) | v_{\nu \nu} (x_0).
	\end{equation}
	Using this in \eqref{Eq:GradSquaredExpNearFB1} and recalling that $|\nabla v(x_0)|=1$, we obtain
	\begin{equation}
		\label{Eq:GradSquaredExpNearFB2}
		|\nabla v(x_0 - t\nu)|^2 = 1 + 2 t v_{\nu \nu} (x_0) + O(t^2).
	\end{equation}
	Combining this expression with \eqref{Eq:vExpNearFB1} we get
	\begin{equation}
		\label{Eq:LaplacianExpNearFB}
		\Delta v(x_0-t\nu) 
		= - \dfrac{\alpha}{2}  \, \dfrac{2 t v_{\nu \nu} (x_0) + O(t^2)}{ t  + O(t^2)} 
		= - \dfrac{\alpha}{2}  \, \dfrac{2  v_{\nu \nu} (x_0) + O(t)}{1 + O(t)},
	\end{equation}
	which gives 
	\begin{equation}
		\label{Eq:LaplacianFreeBoundaryProof}
		\lim_{t\downarrow0}\Delta v(x_0-t\nu) = - \alpha v_{\nu \nu} (x_0).
	\end{equation}

	Finally, at points $x_0 -t\nu$, the mean curvature of the level set of $v$ passing through $x_0 -t\nu$ is given by the expression\footnote{For a smooth function $w$, the mean curvature of the boundary of a sublevel set $\{w \leq k\}$ is usually defined as 
		\begin{equation}
			H := \div\left(\dfrac{\nabla w}{|\nabla w|}\right) = \dfrac{1}{|\nabla w|} \left(\Delta w - w_{\nu \nu} \right),
		\end{equation}
		where $w_{\nu \nu}= w_i w_j w_{ij} |\nabla w|^{-2}$ is the second derivative of $w$ in the direction of $\nu = \nabla w/|\nabla w|$.}
	\begin{equation}
		H(x_0 - t\nu) =  \dfrac{\Delta v(x_0 - t\nu) - v_{\nu \nu}(x_0 - t\nu)}{|\nabla v(x_0 - t\nu)|}. 
	\end{equation}
	Since $|\nabla v(x_0)|=1$, the result follows from \eqref{Eq:LaplacianFreeBoundaryProof}.
\end{proof}

Taking the previous lemma into account, one realizes that the proof of \Cref{Th:StabilityCondition} can be carried out as well in the case $\alpha = 0$, thus obtaining an alternative proof of \eqref{Eq:StabilityOnePhase-Intro} which uses only elementary computations and integration by parts.
Indeed, one can follow the whole proof with $\alpha = 0$ and taking into account that in this case $\Delta v = 0$ in $\{v>0\}\cap \Omega$.
Then, the term $I_{\mathrm{parts}}$ defined in \eqref{Eq:Iparts} is simply
\begin{equation}
	I_{\mathrm{parts}} = 
	- \dfrac{1}{4}\int_{\{v>0\} \cap \Omega} (|\nabla v|^2 +1 ) \div \left[ \left( \dfrac{\nabla (|\nabla v|^2)}{|\nabla v|^4} \right ) \varphi^2 \right] \d x 
\end{equation}
Using that on $\partial \{v>0\} \cap \Omega$ we have $|\nabla v|= 1$ and $\nu = - \nabla v$,  and taking into account the fact that $\partial_i (|\nabla v|^2) = 2 v_j v_{ji}$ (and thus $\nabla (|\nabla v|^2) \cdot \nu = -2 \partial_i (|\nabla v|^2) v_i = -2 v_j v_{ji} v_i = -2 v_{\nu \nu}$), it follows readily that 
\begin{equation}
	I_{\mathrm{parts}}  =
	\int_{\partial \{v>0\} \cap \Omega} v_{\nu \nu } \varphi^2 \d \sigma	
	+ \dfrac{1}{4}\int_{\{v>0\} \cap \Omega}  \dfrac{|\nabla (|\nabla v|^2)|^2}{|\nabla v|^4}  \varphi^2  \d x
\end{equation}
and putting this in \eqref{Eq:2ndVarModifiedAltPhillipsPreInt} and taking into account \Cref{Lemma:MeandCurvatureFB} we readily obtain \eqref{Eq:StabilityOnePhase-Intro}.

A priori, the regimes $\alpha=0$ and $\alpha>0$ seem quite different, at least in terms of the stability inequality.
However, it can be seen that \eqref{Eq:StabilityModifiedAltPhillips1} converges, as $\alpha \to 0$, to the stability inequality for the one phase problem \eqref{Eq:StabilityOnePhase-Intro}.
This link can be made more clear if we write  \eqref{Eq:StabilityModifiedAltPhillips1} as
\begin{equation}
	\label{Eq:StabilityModifiedAltPhillips2}
	\int_{\{v>0\} \cap \Omega} v^{\alpha - 1} \Delta v \varphi^2  \d x \leq \int_{\{v>0\} \cap \Omega} v^\alpha |\nabla \varphi |^2  \d x,
\end{equation}
or
\begin{equation}
	\label{Eq:StabilityModifiedAltPhillips3}
	\alpha \int_{\{v>0\} \cap \Omega} v^{\alpha-1} \dfrac{1 - |\nabla v|^2}{2 v}  \varphi^2  \d x \leq \int_{\{v>0\} \cap \Omega} v^\alpha |\nabla \varphi |^2  \d x.
\end{equation}
On the one hand, it is easy to see that $\alpha v^{\alpha - 1}$ converges in the sense of distributions (for $v$ fixed) to the delta distribution on the free boundary as $\alpha \to 0$.
Indeed, for any $\phi\in C^\infty_c (\Omega)$, using integration by parts we have
\begin{equation}
	\alpha \int_{\{v>0\} \cap \Omega} v^{\alpha-1} \phi \d x 
	=  \int_{\{v>0\} \cap \Omega} \nabla(v^\alpha) \cdot \dfrac{\nabla v }{|\nabla v|^2} \phi \d x  
	= - \int_{\{v>0\} \cap \Omega} v^\alpha \div \left ( \dfrac{\nabla v }{|\nabla v|^2}\phi \right)  \d x.
\end{equation}
Letting $\alpha \to 0$, in the right-hand side we obtain (by the divergence theorem using that $\nu = - \nabla v/|\nabla v|$ and that $|\nabla v |= 1$ on $\partial \{v>0\} \cap \Omega$)
\begin{equation}
	- \int_{\{v>0\} \cap \Omega}\div \left ( \dfrac{\nabla v }{|\nabla v|^2}\phi \right)  \d x = - \int_{\partial \{v>0\} \cap \Omega}  \dfrac{\nabla v }{|\nabla v|^2} \cdot \nu \phi   \d \sigma  = \int_{\partial \{v>0\} \cap \Omega}   \phi   \d \sigma .
\end{equation}
On the other hand, using \Cref{Lemma:MeandCurvatureFB} we can see that, given $x_0\in \partial \{v>0\}\cap \Omega$ a regular free boundary point, then
\begin{equation}
	\dfrac{1 - |\nabla v(x)|^2}{2 v(x)}  = \dfrac{\Delta v(x)}{\alpha} \to \dfrac{H(x_0)}{1 + \alpha } \quad \text{ as } x \to x_0\in \partial \{v>0\}\cap \Omega,
\end{equation}
where $H(x_0)$ is the mean curvature of the boundary of $\{v=0\}$ oriented towards $-\nabla v$.
As a consequence of the previous two facts, if we have a sequence $\{\alpha_k\}$ with $\alpha_k \downarrow 0$ as $k\to +\infty$, and for each $\alpha_k$ we have an associated stable critical point $v_k$ of $E_{\alpha_k}^\mathrm{M} [\cdot]$, if $v_k \to v_\star$ in a suitable sense, then we would obtain the stability inequality \eqref{Eq:StabilityOnePhase-Intro} for $v_\star$ as $k\to+\infty$.

\section{Classification of axially symmetric solutions}
\label{Sec:AxSymSol}

In this section, we establish our classification result on axially symmetric solutions, stated in \Cref{Th:AxiallySymm}.
To be more precise, if we denote points in $\R^n$ by  $x = ( x', x_n)$, with $x'\in \R^{n-1}$, we say that $u$ is axially symmetric if (up to a rotation), it depends only on $|x'|$ and $x_n$, and we will call $\tau := |x'|$.
In particular, we do not require $u$ to be even with respect to $x_n$, and this includes in our setting the cases where $\{u=0\}$ is a one-sided or a two-sided cone with vertex at the origin. 
These cases will be ruled out (except for the half-space case) if $u$ is stable thanks to \Cref{Th:AxiallySymm}, but until such result is proved we will not assume that $\partial \{u>0\}$ is smooth at the origin.

Note that if we write the equation for $u$ in the $(\tau, x_n)$ coordinates we have
\begin{equation}
	\Delta u = u_{\tau \tau}  + \dfrac{n-2}{\tau} u_\tau +  u_{n n} = \dfrac{\gamma}{2} u^{\gamma - 1} \quad \text{ in } \{u>0\} \cap \{\tau > 0 \} \cap\Omega,
\end{equation}
where $u_\tau$ and $u_{\tau \tau}$ denote the first and second derivatives of $u$ with respect to $\tau$.
Differentiating the above equation with respect to $\tau$, we get 
\begin{equation}
	\label{Eq:Equtau}
	\Delta (u_\tau) = \dfrac{n-2}{\tau^2} u_\tau + \dfrac{\gamma (\gamma - 1)}{2} u^{\gamma - 2} u_\tau \quad \text{ in } \{u>0\}\cap \{\tau > 0 \} \cap \Omega,
\end{equation}
which will be used in this section.

To establish \Cref{Th:AxiallySymm}, the crucial point is to obtain the following result:

\begin{proposition}
	\label{Prop:Stabilityutau}
	Let $n\geq 3$ and let $u\geq 0$ be a stable critical point of $E_\gamma^\mathrm{AP} [\cdot]$ that is axially symmetric. 
	Assume that $0 \in \partial\{u>0\}\cap \Omega$, that $\partial\{u>0\}\cap \Omega$ is smooth away from $0$, and that $v= \beta u^{1/\beta}$ is $C^\infty$ up to the free boundary away from $0$.
	
	Then, 
	\begin{equation}
		\label{Eq:Stabilityutau}
		(n-2) \int_{\{u>0\} \cap \Omega}  \dfrac{u_\tau^2}{\tau^2} \eta^2 \d x 
		\leq \int_{\{u>0\} \cap \Omega} u_\tau^2  |\nabla \eta |^2  \d x
	\end{equation}
	or, equivalently,
	\begin{equation}
		\label{Eq:StabilityutauV}
		(n-2) \int_{\{v>0\} \cap \Omega} v^\alpha \dfrac{v_\tau^2}{\tau^2} \eta^2 \d x 
		\leq \int_{\{v>0\} \cap \Omega} v^\alpha v_\tau^2  |\nabla \eta |^2  \d x,
	\end{equation}
	for every $\eta  \in C^1_c(\Omega)$.
\end{proposition}

The main idea to establish \eqref{Eq:Stabilityutau} is to take, in the stability condition, a test function of the form $\varphi = c \eta$, where $\eta$ has compact support and $c$ does not, but satisfies an appropriated linearized equation.
For semilinear equations $-\Delta u = f(u)$, the choice $\varphi = c \eta$ leads, after integration by parts in the stability condition \eqref{Eq:StabilitySemilinear}, to a left-hand side with the term $c(\Delta +f'(u))c$, and then one takes $c= u_\tau$ since $(\Delta + f'(u))u_\tau = (n-2) \tau^{-2} u_\tau$, leading to an inequality analogous to~\eqref{Eq:Stabilityutau}; see~\cite{FernandezRealRosOton2019global} for more details.
In view of the linearized equation \eqref{Eq:Equtau} we would like to choose as well $c= u_\tau$.
However, this choice does not work since in the stability inequality~\eqref{Eq:StabilityAP1} the linearized operator $\Delta - (\gamma/2) (\gamma-1) u^{\gamma - 2}$ does not explicitly appear. 
Indeed, the key point to prove \Cref{Prop:Stabilityutau} is realizing that the right choice of test function consists of taking $c= u^{-\gamma/2} u_\tau$. 
As we will see below in the proof, if we first take $\varphi= u^{-\gamma/2} \xi$ for some test function $\xi$ yet to choose, at least for $\gamma>2/3$ we get
\begin{equation}
	- \dfrac{\gamma (\gamma-1)}{2}   
	\int_{\{u>0\} \cap \Omega}  u^{\gamma - 2} \xi^2  \d x  \leq  \int_{\{u>0\} \cap \Omega} |\nabla \xi|^2 \d x,
\end{equation}
which resembles very much to the stability inequality \eqref{Eq:StabilitySemilinear} for semilinear equations.
From this, we then take $\xi = u_\tau \eta$ obtaining \eqref{Eq:Stabilityutau} for $\gamma > 2/3$.
However, as will be explained in more detail in the proof, this argument does not directly work in the range $\gamma \in (0,2/3)$ due to the presence of two singular free boundary terms appearing after using integration by parts.
To overcome this, it is necessary to do a more careful analysis, using the precise behavior of $u$ and its derivatives up to the free boundary to observe a crucial cancellation.

From the previous reasoning, one realizes that our choice of test function is actually $\varphi = c \eta$ with $c = v_\tau$, where $v = \beta u^{1/\beta}$.
This leads to an alternative way of establishing \Cref{Prop:Stabilityutau}, using the function $v$ and the stability inequality \eqref{Eq:StabilityModifiedAltPhillips1} in a slightly more direct way.
This is how we will present first the proof of \Cref{Prop:Stabilityutau}, establishing \eqref{Eq:StabilityutauV}, and later we will show how to obtain \eqref{Eq:Stabilityutau} from the stability inequality \eqref{Eq:StabilityAP1} as explained in the previous paragraph.
Even if one can really go from \eqref{Eq:Stabilityutau} to \eqref{Eq:StabilityutauV} and back simply using the relation between $u$ and $v$, we think that it is more insightful to show both proofs.

\begin{proof}[Proof of \Cref{Prop:Stabilityutau} using $v$]
	Since $\partial\{u>0\}\cap \Omega$ is smooth away from $0\in \partial\{u>0\}$, the stability condition \eqref{Eq:StabilityModifiedAltPhillips1} holds for any test function $\varphi\in C^1_c(\Omega \setminus  \{0\})$.
	We start by choosing $\varphi$ in \eqref{Eq:StabilityModifiedAltPhillips1} of the form
	\begin{equation}
		\varphi = c \eta,
	\end{equation}
	where $\eta$ is as in the statement of the result, but with support in $\{\tau>0\}$, and where $c$ is a smooth function in $\overline{\{v>0\}}\cap \Omega \cap \{\tau >0\}$ not necessarily with compact support and with $|c|\leq |\nabla v|$.
	In view of \Cref{Remark:AssumptionsStabilityCondition}, this last assumption allows us  to remove the nonvanishing assumption on $|\nabla v|$ in \Cref{Th:StabilityCondition}.
	Note that the free boundary $\partial \{v>0\} \cap \Omega$ is smooth in the support of $\eta$, and thus we have that $v$ is $C^\infty$ up to the free boundary in the support of $\eta$.
	
	With this choice of test function in \eqref{Eq:StabilityModifiedAltPhillips1}, we compute the right-hand side of the resulting stability inequality using integration by parts. 
	We get	
	\begin{equation}
		\begin{split}
			\int_{\{v>0\} \cap \Omega}  v^\alpha |\nabla \varphi|^2 \d x 
			&= \int_{\{v>0\} \cap \Omega}  v^\alpha |\nabla c|^2 \eta^2  \d x + \int_{\{v>0\} \cap \Omega}  v^\alpha c^2 |\nabla \eta|^2  \d x \\
			&\quad \quad + \int_{\{v>0\} \cap \Omega}  v^\alpha (c \nabla c) \cdot  \nabla \eta^2 \d x \\
			& = \int_{\{v>0\} \cap \Omega}  v^\alpha c^2 |\nabla \eta|^2  \d x \\
			&\quad \quad -  \alpha \int_{\{v>0\} \cap \Omega}  v^{\alpha-1} c (\nabla c \cdot  \nabla v ) \eta^2 \d x  -  \int_{\{v>0\} \cap \Omega}  v^\alpha c \Delta c \eta^2 \d x .
		\end{split}
	\end{equation}
	Note that here no boundary term appears since $c$ is smooth up to the free boundary and then $v^\alpha c \nabla c$ vanishes on $\partial \{v>0\} \cap \Omega$.
	As a consequence, we have
	\begin{equation}
		\label{Eq:StabilityvCETA}
		\int_{\{v>0\} \cap \Omega}  v^\alpha \left ( \Delta c + \dfrac{\alpha}{2} \dfrac{ 1 - |\nabla v|^2 }{v^2} c + \alpha\dfrac{1}{v}\nabla c \cdot  \nabla v  \right) c \eta^2 \d x
		\leq \int_{\{v>0\} \cap \Omega}  v^\alpha c^2 |\nabla \eta|^2  \d x.
	\end{equation}

	To obtain \eqref{Eq:StabilityutauV}, we take $c = v_\tau$ (which satisfies the hypothesis mentioned above), and it is enough to check that in $\{v>0\}\cap \{\tau > 0 \} \cap \Omega$ it holds
	\begin{equation}
		\label{Eq:LaplacianCv}
		\Delta v_\tau + \dfrac{\alpha}{2} \dfrac{ 1 - |\nabla v|^2 }{v^2} v_\tau + \alpha\dfrac{1}{v}\nabla (v_\tau) \cdot  \nabla v  =
		(n-2) \dfrac{v_\tau}{\tau^2}.
	\end{equation}
	On the one hand, note that we can rewrite the left-hand side in the equation for $u_\tau$ using that $v = \beta u^{1/\beta}$ and that the gradients of $u$ and $v$ are related through \eqref{Eq:RelationGradientsuv}.
	Thus, \eqref{Eq:Equtau} becomes
	\begin{equation}
		\label{Eq:EqutauV}
		\Delta (u_\tau) =  (\beta v)^{\alpha/2} 
		\left ( \dfrac{n-2}{\tau^2} \, v_\tau + \dfrac{\alpha}{2}  (\gamma - 1)\beta \, \dfrac{v_\tau}{v^2} \right ) \quad \text{ in } \{u>0\}\cap \{\tau > 0 \} \cap \Omega.
	\end{equation}
	On the other hand, from the relation $u_\tau = (\beta v)^{\alpha/2} v_\tau$ we easily obtain
	\begin{equation}
		\Delta (u_\tau) = (\beta v)^{\alpha/2} \left( \frac{\alpha}{2} \left(\dfrac{\alpha}{2} - 1\right) \dfrac{|\nabla v|^2}{v^2} v_\tau
		+ \dfrac{\alpha}{2} \dfrac{\Delta v}{v} v_\tau 
		+ \alpha \dfrac{1}{v} \nabla(v_\tau)\cdot  \nabla v 
		+ \Delta (v_\tau)
		\right),
	\end{equation}
	also in $\{v>0\}\cap \{\tau > 0 \} \cap \Omega$.
	Putting together the last two equations, and using that $(\gamma - 1) \beta = \alpha/2 - 1$ and the equation for $v$, \eqref{Eq:ModAltPhillipsEquation}, we get
	\begin{equation}
		\begin{split}
		\Delta (v_\tau) + \alpha \dfrac{1}{v} \nabla(v_\tau)\cdot  \nabla v 
		&= \dfrac{n-2}{\tau^2} \, v_\tau  
		+ \dfrac{\alpha}{2} \left( (\gamma - 1) \beta - \left(\dfrac{\alpha}{2} - 1\right)|\nabla v|^2 \right) \dfrac{v_\tau }{v^2} - \dfrac{\alpha}{2} \dfrac{\Delta v}{v} v_\tau \\
		&=  \dfrac{n-2}{\tau^2} \, v_\tau  
		+ \dfrac{\alpha}{2} \left( \left(\dfrac{\alpha}{2} - 1\right) \dfrac{1 - |\nabla v|^2 }{v^2 } -  \dfrac{\Delta v}{v} \right)  v_\tau \\
		&=  \dfrac{n-2}{\tau^2} \, v_\tau  
		+ \dfrac{\alpha}{2} \left( \left(\dfrac{\alpha}{2} - 1\right) \dfrac{1 - |\nabla v|^2 }{v^2 } - \dfrac{\alpha}{2} \dfrac{1 - |\nabla v|^2 }{v^2 } \right)  v_\tau ,
		\end{split}
	\end{equation}
	which gives \eqref{Eq:LaplacianCv} and, as a consequence, \eqref{Eq:StabilityutauV} for every $\eta  \in C^1_c(\Omega \cap \{\tau > 0\})$.
	
	To conclude, we need to check that \eqref{Eq:StabilityutauV} holds for $\eta$ not necessarily vanishing on $\{\tau = 0\}$. 
	For this, we will use axially and radially symmetric cut-off functions build from a smooth one-dimensional function $\zeta$ which is $0$ in $(-\infty,1/2)$ and $1$ in $(1,+\infty)$.
	First, for $\varepsilon>0$, replace $\eta$ by $\eta \zeta_\varepsilon$, with $\zeta_\varepsilon(\tau)= \zeta (\tau/\varepsilon)$ and $\eta  \in C^1_c(\Omega \setminus \{0\})$. 
	Then, the stability inequality \eqref{Eq:StabilityutauV} holds for $\eta  \in C^1_c(\Omega \setminus \{0\})$ if we show that the term 
	\begin{equation}
		\int_{-\infty}^{+\infty} \int_{\varepsilon/2}^{\varepsilon} \tau^{n-2} v^\alpha v_\tau^2 \eta^2 |\nabla\zeta_\varepsilon|^2 \d \tau \d x_1
	\end{equation}
	converges to $0$ as $\varepsilon\to 0$.
	This follows, taking into account that $|\nabla\zeta_\varepsilon|^2\leq C \varepsilon^{-2}$, from the fact that $n\geq3$ and that $|v_\tau|\to 0$ as $\tau\to0$ (which is a consequence of the axial symmetry of $v_\tau$ noting that $v$ is smooth in the support of~$\eta$ up to the free boundary).
	Finally, we repeat the same argument replacing the previous $\eta$ by $\eta \zeta_\varepsilon$, where now $\zeta_\varepsilon$ is a radial cut-off function $\zeta_\varepsilon(r)= \zeta (r/\varepsilon)$ with $r=|x|$, and $\eta  \in C^1_c(\Omega)$.
	In this case, to show that \eqref{Eq:StabilityutauV} holds for $\eta  \in C^1_c(\Omega)$ we need to see that the term
	\begin{equation}
		 \int_{\varepsilon/2}^{\varepsilon} r^{n-1} v^\alpha v_\tau^2 \eta^2 |\nabla\zeta_\varepsilon|^2 \d r
	\end{equation}
	converges to $0$ as $\varepsilon\to 0$, which follows from the fact that $v^\alpha \leq C \varepsilon^\alpha$ in $B_\varepsilon$, the support of $\nabla \zeta_\varepsilon$.
\end{proof}

\begin{proof}[Proof of \Cref{Prop:Stabilityutau} using $u$]	
	In this case the proof requires a finer analysis of the behavior of certain quantities near the free boundary.
	We will assume first that $\gamma > 2/3$ and at the end of the proof (Step 3 below) we will show how to proceed in the range $\gamma \in (0,2/3]$.
	We will only establish \eqref{Eq:Stabilityutau} for $\eta  \in C^1_c(\Omega \cap \{\tau > 0\})$, since then we can use the same arguments as in the end of the previous proof to obtain the result for $\eta  \in C^1_c(\Omega)$.
	
	\underline{Step 1:} Assume $\gamma > 2/3$ (that is, $\alpha >1$). 
	In the stability condition \eqref{Eq:StabilityAP1} we take	
	\begin{equation}
		\varphi = u^{-\gamma/2} \xi,
	\end{equation}
	with $\xi$ a $C^1$ function in $\{u>0\}\cap \Omega$, with compact support in $\Omega$, such that $\xi \equiv 0$ in $\{u=0\}\cap \Omega$, and satisfying $|\xi| \asymp \dfb^{\frac{\gamma}{2 - \gamma}} = \dfb^\frac{\alpha}{2}$.
	Note that these assumptions make $\varphi$ an admissible choice, since the behavior of $\xi$ near the free boundary compensates the singularity of $u^{-\gamma/2}$.
	
	Now, since $\varphi_j = - (\gamma/2) u^{-\gamma/2-1} u_j \xi + u^{-\gamma/2} \xi_j$ for $j=1,\ldots,n$,	we have
	\begin{equation}
		|\nabla \varphi|^2 = \left(\frac{\gamma}{2}\right)^2 u^{-\gamma-2} |\nabla u|^2 \xi^2 
		- \frac{\gamma}{2} u^{-\gamma-1} \nabla u \cdot \nabla (\xi^2)  
		+ u^{-\gamma} |\nabla \xi|^2.
	\end{equation}
	Thus, the stability condition \eqref{Eq:StabilityAP1} becomes
	\begin{equation}
		\begin{split}
			\dfrac{2-\gamma}{2} \dfrac{\gamma}{2}  
			\int_{\{u>0\} \cap \Omega} \dfrac{u^\gamma - |\nabla u|^2}{u^2}\xi^2  \d x 
			&\leq 
			\left(\frac{\gamma}{2}\right)^2 \int_{\{u>0\} \cap \Omega} u^{-2} |\nabla u|^2 \xi^2 \d x
			- \frac{\gamma}{2} \int_{\{u>0\} \cap \Omega} u^{-1} \nabla u \cdot \nabla (\xi^2) \d x \\
			& \quad \quad + \int_{\{u>0\} \cap \Omega} |\nabla \xi|^2 \d x.
		\end{split}
	\end{equation}
	Integrating by parts in the second term of the right-hand side we obtain
	\begin{equation}
		\label{Eq:INTBYPARTS1}
		- \frac{\gamma}{2} \int_{\{u>0\} \cap \Omega} u^{-1} \nabla u \cdot \nabla (\xi^2) \d x = \frac{\gamma}{2} \int_{\{u>0\} \cap \Omega} \div \left( \dfrac{\nabla u}{u} \right) \xi^2 \d x.
	\end{equation}
	Note that here no boundary term appears if $\alpha > 1$ (that is, $\gamma > 2/3$), since  $u^{-1} \nabla u \asymp \dfb^{-1}$ and therefore $u^{-1} \nabla u \xi ^2 \asymp \dfb^{\frac{2\gamma}{2 - \gamma}-1} = \dfb^{\alpha - 1}$.
	Using the equation \eqref{Eq:AltPhillipsEquation}, we have
	\begin{equation}
		\dfrac{\gamma}{2} \div \left( \dfrac{\nabla u}{u} \right) 
		= \dfrac{\gamma}{2} \dfrac{u \Delta u - |\nabla u|^2}{u^2}
		= \left( \dfrac{\gamma}{2} \right)^2 u^{\gamma - 2} -\dfrac{\gamma}{2} \dfrac{|\nabla u|^2}{u^2} \quad \text{ in } \{u>0\}\cap \Omega
	\end{equation}
	and thus the stability condition becomes
	\begin{equation}
		\begin{split}
			\dfrac{2-\gamma}{2} \dfrac{\gamma}{2}  
			\int_{\{u>0\} \cap \Omega} & u^{\gamma - 2} \xi^2  \d x 
			- \dfrac{2-\gamma}{2} \dfrac{\gamma}{2}  
			\int_{\{u>0\} \cap \Omega} \dfrac{|\nabla u|^2}{u^2}\xi^2  \d x \\
			& \leq 
			\left(\frac{\gamma}{2}\right)^2 \int_{\{u>0\} \cap \Omega} \dfrac{|\nabla u|^2}{u^2} \xi^2 \d x
			+ \left( \dfrac{\gamma}{2} \right)^2 
			\int_{\{u>0\} \cap \Omega} u^{\gamma - 2} \xi^2 \d x \\ & \quad \quad -\dfrac{\gamma}{2} \int_{\{u>0\} \cap \Omega}  \dfrac{|\nabla u|^2}{u^2} \xi^2 \d x 
			+ \int_{\{u>0\} \cap \Omega} |\nabla \xi|^2 \d x.
		\end{split}
	\end{equation}
	Rearranging terms we have
	\begin{equation}
		\label{Eq:StabilityAP2}
		\dfrac{\gamma (1 - \gamma)}{2}   
		\int_{\{u>0\} \cap \Omega}  u^{\gamma - 2} \xi^2  \d x  \leq  \int_{\{u>0\} \cap \Omega} |\nabla \xi|^2 \d x.
	\end{equation}
	Note that since we chose $\xi$ such that $|\xi|^2 \asymp \dfb^{\frac{2\gamma}{2 - \gamma}} = \dfb^\alpha$, the left-hand side is integrable if and only if $\alpha > 1$ (i.e., $\gamma > 2/3$).

	\underline{Step 2:}  In the previous inequality \eqref{Eq:StabilityAP2} we take a test function of the form
	\begin{equation}
		\xi = c \eta,
	\end{equation}
	where $c = u_\tau$ and $\eta$ is a $C^1$ function with compact support in $\Omega \cap \{\tau > 0\}$ not necessarily vanishing on the free boundary.
	Note that $\xi$ satisfies the requirements of behavior near the free boundary since $|u_\tau| \asymp \dfb^\frac{\alpha}{2}$.

	With this choice of $\xi$ we have
	\begin{equation}
		 \int_{\{u>0\} \cap \Omega} |\nabla \xi|^2 \d x 
		=   \int_{\{u>0\} \cap \Omega} c^2 |\nabla \eta|^2 \d x
		+  \int_{\{u>0\} \cap \Omega} c \nabla c \cdot \nabla (\eta^2) \d x
		+  \int_{\{u>0\} \cap \Omega} |\nabla c|^2 \eta^2 \d x.
	\end{equation}
	Integrating the second term by parts we get 
	\begin{equation}
		\label{Eq:INTBYPARTS2}
		\int_{\{u>0\} \cap \Omega} c \nabla c \cdot \nabla (\eta^2) \d x = - \int_{\{u>0\} \cap \Omega} |\nabla c|^2 \eta^2 \d x - \int_{\{u>0\} \cap \Omega} c \Delta c \eta^2 \d x.
	\end{equation}
	Note that there are no boundary terms here if $c = u_\tau$, since $|c \nabla c| \leq \dfb^{\alpha - 1}$ and $\alpha > 1$ (i.e., $\gamma > 2/3$).
	As a consequence, we have 
	\begin{equation}
		\int_{\{u>0\} \cap \Omega} |\nabla \xi|^2 \d x 
		=   \int_{\{u>0\} \cap \Omega} c^2 |\nabla \eta|^2 \d x
		- \int_{\{u>0\} \cap \Omega} c \Delta c \eta^2 \d x.
	\end{equation}
	Using the equation for $c=u_\tau$, \eqref{Eq:Equtau}, we get
	\begin{equation}
		- \int_{\{u>0\} \cap \Omega} c \Delta c \eta^2 \d x 
		= 
		- (n-2) \int_{\{u>0\} \cap \Omega} \dfrac{u_\tau^2}{\tau^2} \eta^2 \d x
		- \dfrac{\gamma (\gamma - 1)}{2} \int_{\{u>0\} \cap \Omega} u^{\gamma - 2} u_\tau^2 \eta^2 \d x,
	\end{equation}
	and putting all together we finally obtain  \eqref{Eq:Stabilityutau} for $\gamma > 2/3$.

	\underline{Step 3:} We show how to proceed if $\gamma \leq 2/3$. 
	The main concern in this range is that, in the computations done before, there are two points in which, when integrating by parts, a boundary term would appear, and each of these terms would be infinite if $\gamma < 2/3$. 	
	The rough explanation of why this will not be eventually an issue is that the two terms are infinities of the same order that cancel out, leading to \eqref{Eq:Stabilityutau}.
	
	To formalize these thoughts, we follow the computations of the previous steps but computing all the integrals in $\{u>\delta\}$ for some $\delta > 0$.
	If we do it in this way, we should check which are the boundary terms that appear in the integration by parts.
	First, in \eqref{Eq:INTBYPARTS1} we would get the boundary term
	\begin{equation}
		\label{Eq:FBTERM1}
		-\frac{\gamma}{2} \int_{\partial \{u>\delta\} \cap \Omega} \left( u^{-1} \nabla u \xi^2 \right ) \cdot \nu\d \sigma = -\frac{\gamma}{2} \int_{\partial \{u>\delta\} \cap \Omega} \left( u^{-1} \nabla u  u_\tau^2 \eta^2 \right ) \cdot \nu\d \sigma 
	\end{equation}
	in the right-hand side of the stability inequality.
	Second, in \eqref{Eq:INTBYPARTS2} we would get
	\begin{equation}
		\int_{\partial \{u>\delta \} \cap \Omega} \left (c \nabla c  \eta^2 \right) \cdot \nu  \d \sigma 
		=\int_{\partial \{u>\delta \} \cap \Omega} \left (u_\tau \nabla u_\tau  \eta^2 \right) \cdot \nu  \d \sigma  
	\end{equation}
	also in the right-hand side.
	
	Summarizing, we would obtain
	\begin{equation}
		\label{Eq:StabilityutauDELTA}
		(n-2) \int_{\{u>\delta\} \cap \Omega}  \dfrac{u_\tau^2}{\tau^2} \eta^2 \d x 
		\leq \int_{\{u>\delta\} \cap \Omega} u_\tau^2  |\nabla \eta |^2  \d x
		+ \int_{\partial \{u>\delta \} \cap \Omega} \eta^2 u_\tau  \left( \nabla u_\tau   -\frac{\gamma}{2} u^{-1} \nabla u  u_\tau   \right) \cdot \nu  \d \sigma .
	\end{equation}
	To get \eqref{Eq:Stabilityutau} we want to let $\delta \to 0$ in the previous inequality and use dominated convergence, and thus we have to prove that the boundary term goes to zero as $\delta \to 0$.
	We realize that 
	\begin{equation}
		\begin{split}
			u_\tau  \left( \nabla u_\tau   -\frac{\gamma}{2} u^{-1} \nabla u  u_\tau   \right) 
			&= u_\tau u^{-\gamma/2} \left( \nabla u_\tau u^{\gamma/2} -u_\tau   \frac{\gamma}{2} u^{\gamma/2 - 1} \nabla u    \right) \\
			&= u_\tau u^{-\gamma/2} \left( \nabla u_\tau u^{\gamma/2} -u_\tau  \nabla (u^{\gamma/2})    \right) \\
			&= u^\gamma \dfrac{u_\tau}{u^{\gamma/2}} \nabla \left(\dfrac{u_\tau}{u^{\gamma/2}}\right)\\
			&= u^\gamma v_\tau \nabla (v_\tau),
		\end{split}
	\end{equation}
	where in this last equality we have used the relation between the gradients of $u$ and $v$ given by~ \eqref{Eq:RelationGradientsuv}.
	Since $v$ is $C^\infty$ up to the free boundary in the support of $\eta$, it follows that $|u^\gamma v_\tau \nabla (v_\tau)| \leq C \delta^\alpha$ in $\partial \{u>\delta \} \cap \Omega$ for some constant $C$ depending on $\eta$ but independent of~$\delta$.
	From this the result follows letting $\delta \to 0$ (choosing a sequence $\delta_k\to$ for which $|\partial \{u>\delta_k \} \cap \Omega|\leq C$, which exists by Sard's theorem).
\end{proof}

With \Cref{Prop:Stabilityutau} proved we can now establish \Cref{Th:AxiallySymm}.

\begin{proof}[Proof of \Cref{Th:AxiallySymm}]
	First, note that by a simple approximation argument, we can take the test function~$\eta$ in \eqref{Eq:Stabilityutau} being just Lipschitz and not $C^1$.
	Taking this into account, let $R \geq 1$, $\varepsilon\in (0,1)$, and $\theta > 0$ a constant to be chosen later, and for such parameters we take in \eqref{Eq:Stabilityutau}, with $\Omega = \R^n$, the Lipschitz function
	\begin{equation}
		\eta = \begin{cases}
			\tau^{-\theta/2} \zeta_R & \text{ for } \tau > \varepsilon,\\
			\varepsilon^{-\theta/2} \zeta_R & \text{ for } \tau \leq \varepsilon.
		\end{cases}
	\end{equation}
	Here $\zeta_R$ is a cut-off function such that $\zeta_R \equiv 1$ in $B_R$ and $\zeta_R \equiv 0$ in $\R^n \setminus B_{2R}$.
	
	Our goal will be to estimate some integrals, and through the proof we will denote by $C$ a positive constant depending on $n$, $\gamma$, $\theta$, and $u$ (but independent of $\varepsilon$ and $R$) which may change each time it appears.
	In particular, we will use that $|\nabla \zeta_R|\leq C/R$ in $B_{2R} \setminus B_R$ and that, by $C^{1, \beta-1}$ global regularity (see \Cref{Sec:PreliminaryResults}), we have
	\begin{equation}
		\label{Eq:BoundutauR}
		u_\tau^2 \leq C R^{2\beta - 2} = C R^\alpha \quad \text{ in } B_{2R} \setminus B_R.
	\end{equation}

	First, we analyze the term integrated on $\{\tau \leq \varepsilon \}$ in the right-hand side of \eqref{Eq:Stabilityutau}.
	We have
	\begin{equation}
		\begin{split}
			\int_{\{u>0\}  \cap \{\tau \leq \varepsilon \} \cap B_{2R}} u_\tau^2  |\nabla \eta |^2  \d x 
			&=  \varepsilon^{-\theta} 	\int_{\{u>0\}  \cap \{\tau \leq \varepsilon \} \cap (B_{2R} \setminus B_R) } u_\tau^2  |\nabla \zeta_R |^2  \d x  \\
			& \leq C R^{\alpha - 1} \varepsilon^{-\theta}  \int_0^\varepsilon \tau^{n - 2} \d \tau \\
			&\leq C R^{\alpha - 1} \varepsilon^{n - 1 - \theta}.
		\end{split}
	\end{equation}
	Thus, from \eqref{Eq:Stabilityutau} we get
	\begin{equation}
		(n-2) \int_{\{u>0\}\cap \{\tau > \varepsilon \}\cap B_{2R}} \tau^{-\theta - 2}  u_\tau^2  \zeta_R^2 \d x 
		\leq \int_{\{u>\varepsilon\} \cap \{\tau > \varepsilon \} \cap B_{2R}} u_\tau^2  |\nabla (\tau^{-\theta/2} \zeta_R) |^2  \d x 
		+ C R^{\alpha - 1} \varepsilon^{n - 1 - \theta}.
	\end{equation}
	Now, for every $\delta >0$ (which will be chosen later depending only on $n$ and $\theta$) we have that
	\begin{equation}
		|\nabla (\tau^{-\theta/2} \zeta_R) |^2 \leq (1 + \delta) \dfrac{\theta^2}{4} \tau^{-\theta - 2}  \zeta_R^2  + \left(1 + \dfrac{1}{\delta}\right)\tau^{-\theta }  |\nabla \zeta_R|^2 \chi_{B_{2R} \setminus B_{R}},
	\end{equation}
	and thus we obtain
	\begin{equation}
		\label{Eq:ProofAxiallySymmEst}
		\begin{split}
			\left(n - 2 - (1 + \delta)\dfrac{\theta^2}{4} \right) \int_{\{u>0\} \cap \{\tau > \varepsilon \} \cap B_{2R}} \tau^{-\theta - 2}  u_\tau^2  \zeta_R^2 \d x 
			&\leq  
			C\int_{\{u>0\} \cap \{\tau > \varepsilon \} \cap (B_{2R} \setminus B_R)}  u_\tau^2 \tau^{-\theta }  |\nabla \zeta_R|^2 \d x \\
			& \quad \quad + C R^{\alpha - 1} \varepsilon^{n - 1 - \theta}.
		\end{split}
	\end{equation}
	Last, we estimate the first term in the right-hand side as follows:
	\begin{equation}
		\int_{\{u>0\} \cap \{\tau > \varepsilon \} \cap (B_{2R} \setminus B_R)}  u_\tau^2 \tau^{-\theta }  |\nabla \zeta_R|^2 \d x 
		\leq  C R^{\alpha -1} \int_{\varepsilon}^{R} \tau^{n - 2 - \theta} \d \tau
		\leq  C R^{\alpha -1} \left(R^{n - 1 - \theta} - \varepsilon^{n - 1 - \theta}\right),
	\end{equation}
	and choosing a bigger constant $C$ in the second right-hand side term of \eqref{Eq:ProofAxiallySymmEst} if needed, we get
	\begin{equation}
		\left(n - 2 - (1 + \delta)\dfrac{\theta^2}{4} \right) \int_{\{u>0\} \cap \{\tau > \varepsilon \} \cap B_{2R}} \tau^{-\theta - 2}  u_\tau^2  \zeta_R^2 \d x  
		\leq C R^{n + \alpha -  2 - \theta} + C R^{\alpha - 1} \varepsilon^{n - 1 - \theta}.
	\end{equation}
	
	From this last estimate we want to deduce that $u$ is one-dimensional.
	To do it, note that if we take $\theta$ such that
	\begin{equation}
		\label{Eq:ConditionTheta1}
		n - 2 > \dfrac{\theta^2}{4},
	\end{equation}
	after choosing $\delta$ small enough (depending on $n$ and $\theta$ so that $n - 2 > (1 + \delta) \theta^2/4$) we have
	\begin{equation}
		\int_{\{u>0\}  \cap \{\tau > \varepsilon \}  \cap B_{R}} \tau^{-\theta - 2}  u_\tau^2  \d x  
		\leq C R^{n + \alpha -  2 - \theta} + C R^{\alpha - 1} \varepsilon^{n - 1 - \theta}.
	\end{equation}
	Then, taking into account that \eqref{Eq:ConditionTheta1} yields $n - 1 - \theta >0$ (since $n-1 = n-2 + 1 > \theta^2/4 + 1 \geq \theta$), by letting $\varepsilon\to 0$ we obtain
	\begin{equation}
		\int_{\{u>0\} \cap B_{R}} \tau^{-\theta - 2}  u_\tau^2  \d x  
		\leq C R^{n + \alpha -  2 - \theta},
	\end{equation}
	and thus, if we take $\theta$ such that
	\begin{equation}
		\label{Eq:ConditionTheta2}
		n + \alpha -  2 < \theta,
	\end{equation}
	by letting $R\to +\infty$ we deduce that $u_\tau \equiv 0$ in $\R^n$, obtaining that $u$ is one-dimensional (only depending on $x_n$).
	
	Summarizing, if we can choose $\theta>0$ satisfying \eqref{Eq:ConditionTheta1} and \eqref{Eq:ConditionTheta2}, i.e. $n + \alpha -  2 <\theta<2\sqrt{n-2}$,  we can show that $u$ is one-dimensional.
	This can be done in dimensions $n\geq 3$ satisfying
	\begin{equation}
		\label{Eq:ConditionDimensionTheta}
		2\sqrt{n - 2} > n - 2 + \alpha.
	\end{equation}
	Setting $\lambda := \sqrt{n - 2}$, this is equivalent to
	\begin{equation}
		\lambda^2 - 2 \lambda + \alpha < 0,
	\end{equation}
	which will have solutions if the equation $\lambda^2 - 2 \lambda + \alpha = 0$ has two different real roots.
	Denoting them by $\lambda_\pm$ we have
	\begin{equation}
		\lambda_\pm = \dfrac{2 \pm \sqrt{ 4 - 4 \alpha}}{2} = 1 \pm \sqrt{1 - \alpha},
	\end{equation}
	and thus we require $\alpha < 1$.
	For such values of $\alpha$ (that is, for $\gamma<2/3$), the admissible dimensions for which \eqref{Eq:ConditionDimensionTheta} holds are given by
	\begin{equation}
		1 - \sqrt{1 - \alpha} < \sqrt{n-2} < 1 + \sqrt{1 - \alpha},
	\end{equation}
	that is, \eqref{Eq:DimensionConstrain}.
\end{proof}

\begin{remark}
	\label{Remark:ThAxiallySymWithv}
	Note that the last proof can be carried out from the inequality \eqref{Eq:StabilityutauV}, using $v$ instead of $u$.
	Indeed, the proof would follow exactly the same lines with the same test function, but instead of using \eqref{Eq:BoundutauR} to estimate the growth of $u_\tau$, we would use that $v$ is globally Lipschitz and thus $|v_\tau|\leq C$ and $v^\alpha \leq C R^\alpha$.
\end{remark}

\appendix

\section{}
\label{Appendix:ExpansionDet}

\begin{lemma}
	\label{Lemma:ExpansionDet}
	Let $A$ be a square matrix.
	Then, we have the following expansion:
	\begin{equation}
		\label{Eq:DetExpansionOrder2}
		\det(Id+\varepsilon A) = 1 + \varepsilon \Tr A + \varepsilon^2 \dfrac{(\Tr A)^2 -  \Tr (A^2)}{2} + O(\varepsilon^3)
	\end{equation}
\end{lemma}

\begin{proof}
	First, we claim that
	\begin{equation}
		\label{Eq:DetExpansionOrder1}
		\det(Id+\varepsilon A) = 1 + \varepsilon \Tr A + \varepsilon^2 f(A) + O(\varepsilon^3)
	\end{equation}
	for some function $f$.	
	To see this, it suffices to denote by $A_j$ the column vectors of the matrix $A$ and write
	$$
	\det (Id + \varepsilon A) = \det (e_1 + \varepsilon A_1, e_2 + \varepsilon A_2, \ldots, e_n + \varepsilon A_n),
	$$
	where $n$ is the dimension of the matrix and $e_j$ are the canonical basis vectors of $\R^n$.
	Then, since the determinant is multilinear on the columns, we readily see that
	\begin{equation}
		\begin{split}
			\det (Id + \varepsilon A) 
			&= \det (e_1 + \varepsilon A_1, e_2 + \varepsilon A_2, \ldots, e_n + \varepsilon A_n) \\
			&= \det (e_1, e_2, \ldots, e_n) + \varepsilon \big[ \det (A_1, e_2, \ldots, e_n) + \ldots + \det(e_1, \ldots, e_{n-1}, A_n) \big] \\
			& \quad \quad  + \varepsilon^2 f(A) +O(\varepsilon^3) ,
		\end{split}
	\end{equation}
	proving \eqref{Eq:DetExpansionOrder1}.
	
	It remains to find the explicit expression of $f$.
	First, by a standard property of the exponential of a matrix we have $\det[\exp (\varepsilon A)] = \exp [\Tr (\varepsilon A)]$, and thus the expansion of the exponential function gives
	$$
	\det ( Id +  \varepsilon A + \varepsilon^2 A^2 /2 + O(\varepsilon^3)) = 1 + \varepsilon \Tr A + \varepsilon^2 (\Tr A)^2 / 2 + O(\varepsilon^3).
	$$
	Now, taking the left hand-side of the last expression and factoring one $\varepsilon$, using \eqref{Eq:DetExpansionOrder1} we obtain
	\begin{equation}
		\begin{split}
			\det ( Id +  \varepsilon A + \varepsilon^2 A^2 /2 + O(\varepsilon^3))
			&= \det ( Id +  \varepsilon [A + \varepsilon A^2 /2 + O(\varepsilon^2)]) \\
			&= 1 + \varepsilon \Tr (A + \varepsilon A^2 /2 + O(\varepsilon^2)) \\
			& \quad \quad + \varepsilon^2 f(A + \varepsilon A^2 /2 + O(\varepsilon^2)) + O(\varepsilon^3) \\
			&= 1 + \varepsilon \Tr A + \varepsilon^2 \Tr (A^2)/2 + \varepsilon^2 f(A) + O(\varepsilon^3).
		\end{split}
	\end{equation}
	Comparing the second order terms of the two previous expressions we get
	$$
	\Tr (A^2)/2 + f(A) = (\Tr A)^2 / 2,
	$$
	concluding the proof.
\end{proof}

\begin{lemma}
	\label{Lemma:ExpansionNormSqMatrixVector}
	Let $A$ and $B$ be two square matrices and, for $\varepsilon>0$, let $M_\varepsilon$ be a square matrix defined by the expansion
	$$
	M_\varepsilon = Id + \varepsilon A + \varepsilon^2 B + O(\varepsilon^3).
	$$
	
	Then, for any vector $q$ we have
	\begin{equation}
		\label{Eq:ExpansionNormSqMatrixVector}
		|M_\varepsilon^\intercal q|^2 = |q|^2 +  \varepsilon 2 q \cdot A q + \varepsilon^2 (|A^\intercal q|^2 + 2 q \cdot  B q) + O(\varepsilon^3),
	\end{equation}
	where $A^\intercal$ denotes the transpose of the matrix $A$.
\end{lemma}

\begin{proof}
	It is a simple computation, which we show here considering vectors as $n\times 1$ matrices and using matrix notation for the dot product (i.e., $q \cdot q = q^\intercal q$).
	We have $|M_\varepsilon^\intercal q|^2 = (M_\varepsilon^\intercal q)^\intercal M_\varepsilon^\intercal q = q^\intercal M_\varepsilon M_\varepsilon^\intercal q$ 	and using that
	\begin{equation}
		\begin{split}
			M_\varepsilon M_\varepsilon^\intercal 
			&= (Id + \varepsilon A + \varepsilon^2 B)(Id + \varepsilon A^\intercal + \varepsilon^2 B^\intercal) + O(\varepsilon^3) \\
			&= Id + \varepsilon( A + A^\intercal)  + \varepsilon^2 (B + B^\intercal + A A^\intercal) + O(\varepsilon^3),
		\end{split}	
	\end{equation} 
	we obtain
	\begin{equation}
		|M_\varepsilon^\intercal q|^2 = q^\intercal M_\varepsilon M_\varepsilon^\intercal q = |q|^2 +  \varepsilon 2 q^\intercal A q + \varepsilon^2 (|A^\intercal q|^2 + 2 q^\intercal B q) + O(\varepsilon^3).
	\end{equation}
\end{proof}


\end{document}